\documentclass[11pt]{article}
\usepackage{amssymb,a4wide,latexsym,makeidx,epsfig}
\usepackage{amsthm}
\usepackage{dsfont}
\usepackage{amsmath}
\usepackage{lipsum}
\usepackage{enumerate}
\usepackage{mathrsfs}
\usepackage{xcolor}
\usepackage{tikz}
\usepackage{setspace}
\usepackage{geometry}
\usepackage{appendix}
\usepackage{multirow}
\usepackage{subcaption}
\usepackage[colorlinks=true, linkcolor=black, anchorcolor=black, citecolor=black, urlcolor=black, CJKbookmarks=true]{hyperref}
\usepackage{microtype}
\usepackage{colonequals}
\allowdisplaybreaks

\AtBeginDocument{
  \setlength{\abovedisplayskip}{6pt} 
  \setlength{\belowdisplayskip}{6pt} 
  \setlength{\abovedisplayshortskip}{0pt} 
  \setlength{\belowdisplayshortskip}{3pt} 
  \setlength{\jot}{4pt}
}

\newtheorem{theorem}{Theorem}[section]

\newtheorem{lemma}[theorem]{Lemma}
\newtheorem{fact}[theorem]{Fact}
\newtheorem{observation}[theorem]{Observation}
\newtheorem{proposition}[theorem]{Proposition}
\newtheorem{corollary}[theorem]{Corollary}
\newtheorem{problem}[theorem]{Problem}

\newtheorem{conjecture}[theorem]{Conjecture}
\newtheorem{claim}{Claim}[section]

\numberwithin{equation}{section}

\begin{document}

\title{Constrained Ramsey numbers for rainbow $P_5$}

\author{
Xihe Li\footnote{School of Mathematics and Statistics, Shaanxi Normal University, Xi'an, Shaanxi 710119, China.}~\footnote{Corresponding author.}~~~~~~
Xiangxiang Liu\footnote{College of Science, Northwest A\&F University, Yangling, Shaanxi 712100, China.}~~~~~~
}
\date{}
\maketitle
\newcommand\blfootnote[1]{%
\begingroup
\renewcommand\thefootnote{}\footnote{#1}%
\addtocounter{footnote}{-1}%
\endgroup
}
\blfootnote{E-mail addresses: xiheli@snnu.edu.cn, xxliumath@163.com.}
\begin{center}
\begin{minipage}{140mm}
\vskip 0.3cm
\begin{center}
{\small {\bf Abstract}}
\end{center}
{\small
For any graph $H$, let $R_k(H)$ denote the minimum integer $n$ such that in every $k$-edge-coloring of $K_{n}$, there is a monochromatic copy of $H$. Given two graphs $H$ and $G$, the \emph{constrained Ramsey number} $f(H,G)$ is the minimum integer $n$ such that in every edge-coloring of $K_{n}$ with any number of colors, there is either a monochromatic copy of $H$ or a rainbow copy of $G$. Let $P_t$ be the path on $t$ vertices. Gy\'{a}rf\'{a}s, Lehel and Schelp proved that $f(H,P_5)=R_3(H)$ when $H$ is a path, a cycle, or a connected non-bipartite graph. Li, Besse, Magnant, Wang and Watts conjectured that $f(H,P_5)=R_3(H)$ for any graph $H$, and confirmed this for every connected graph and every bipartite graph. In this paper, we address this conjecture for several classes of disconnected graphs with chromatic number at least 3. To the best of our knowledge, our general results encompass all previously known results of this type. We also obtain several results for a bipartite variant of the problem. In addition, we propose a series of related questions from several directions for further research. Our proofs combine the structural characterization of Thomason and Wagner with Simonovits' decomposition-family method and Ramsey numbers for families of graphs.

\vskip 0.1in \noindent {\bf Mathematics Subject Classification (2020)}: \ 05C55; 05D10; 05C35
\vskip 0.1in \noindent {\bf Keywords}: \ constrained Ramsey number; rainbow Ramsey number; rainbow subgraph; Gallai--Ramsey number
}
\end{minipage}
\end{center}

\section{Introduction}
\label{sec:introduction}

Ramsey theory studies the conditions under which specific patterns become unavoidable in host graphs with arbitrary edge-colorings.
Classical Ramsey problems, dating back to Ramsey's foundational theorem~\cite{Ram}, concern the existence of monochromatic subgraphs, that is, subgraphs whose edges all have the same color.
In 1950, Erd\H{o}s and Rado \cite{ErRa} proved the following {\it Canonical Ramsey Theorem}.
Here an edge-colored graph is called {\it rainbow} if all edges are colored differently,
and {\it lexical} if there is a total order of its vertices such that two edges have the same color if and only if they share the same larger endpoint.

\begin{theorem}[Erd\H{o}s-Rado Canonical Ramsey Theorem~\cite{ErRa}]\label{th:canonical}
For any positive integer $p$, there exists an integer $n$ such that, in every edge-coloring of the complete graph $K_n$ with any number of colors, there is a monochromatic, a rainbow, or a lexically colored $K_p$.
\end{theorem}

Motivated by the Canonical Ramsey Theorem, Eroh~\cite{Eroh}, Jamison, Jiang and Ling~\cite{JaJL}, and independently, Chen, Schelp and Wei~\cite{ChSW}, introduced the following notion.
Given two graphs $H$ and $G$, the {\it constrained Ramsey number} (also called {\it rainbow Ramsey number}) $f(H,G)$ is defined as the minimum integer $n$ such that, in every edge-coloring of $K_{n}$ with any number of colors, there is either a monochromatic copy of $H$ or a rainbow copy of $G$.
It follows from Theorem~\ref{th:canonical} that $f(H,G)$ exists if and only if either $H$ is the union of a star and some isolated vertices or $G$ is acyclic (see also \cite{Eroh,JaJL}).
Since its inception, this concept has spurred extensive research into analogous problems across diverse combinatorial structures, including bipartite graphs~\cite{BGLS,ErOe,WaLi}, hypergraphs~\cite{CaSS,LiWa,Liu}, ordered graphs~\cite{GIMSS}, random graphs~\cite{BHHLM,CKMM}, Euclidean space~\cite{FGSXX,GeST} and Boolean lattices~\cite{CCLL}.
In addition, the study of rainbow-colored graphs can be traced back to Euler's 1782 Latin square decomposition problem.
In fact, many classical problems in combinatorics can be reformulated as the search for specific rainbow substructures within edge-colored graphs,
such as Ringel's conjecture~\cite{MoPS21}, the Caccetta-H\"{a}ggkvist conjecture~\cite{AhDH} and graph decomposition problems~\cite{GKMO}.
Moreover, rainbow generalizations of Ramsey-type problems~\cite{FoGP,LiBW}, Tur\'{a}n-type problems~\cite{LiMZ,KMSV} and saturation problems~\cite{LiMX} have been extensively studied in the last two decades.

In 2003, Jamison, Jiang and Ling~\cite{JaJL} initiated the study of $f(S,T)$, where both $S$ and $T$ are trees.
They proved that $f(S,T)=\Omega(st)$ and $f(S,T)= O(st\cdot r(T))$, where $s=|E(S)|$, $t=|E(T)|$ and $r(T)$ is the radius of $T$.
Let $P_t$ be the path on $t$ vertices.
Jamison, Jiang and Ling asked whether $f(S,T)$ is maximized by $f(P_{s+1}, P_{t+1})$.
In 2007, Gy\'{a}rf\'{a}s, Lehel and Schelp~\cite{GyLS} showed that for $t+1\in \{4,5\}$, the answer is negative.
Moreover, Jamison, Jiang and Ling~\cite{JaJL} conjectured that $f(S,T)= O(st)$.
In 2009, Loh and Sudakov~\cite{LoSu} proved that $f(S, P_t) = O(st \log t)$, which matches the lower bound up to a logarithmic factor.
This result was recently improved by Gishboliner, Milojevi\'{c}, Sudakov and Wigderson~\cite{GMSW}, who proved a nearly optimal upper bound for $f(S, P_t)$ which differs from the lower bound by a function of inverse-Ackermann type.
Currently, $f(H, P_t)$ remains an important case of the constrained Ramsey problem.

For $t\in \{4,5\}$, the constrained Ramsey number $f(H, P_t)$ exhibits a surprising connection to the classical Ramsey number for $H$.
For a graph $F$, we refer to a mapping $c: E(F) \to \{1, 2, \ldots, k\}$ as a {\it $k$-edge-coloring} (not necessarily a proper edge-coloring) of $F$.\footnote{The mapping need not be surjective, so we do not require all the $k$ colors to be used.}
Given a graph $H$ and a positive integer $k$, the {\it $k$-color Ramsey number} $R_k(H)$ is the minimum integer $n$ such that in every $k$-edge-coloring of the complete graph $K_{n}$, there is a monochromatic copy of $H$.
Gy\'{a}rf\'{a}s, Lehel and Schelp~\cite{GyLS} proved that $f(H, P_4)=R_2(H)$ for any graph $H$ of order at least $5$, and $f(H,P_5)=R_3(H)$ when $H$ is a path, a cycle, or a connected non-bipartite graph.
Note that for any graph $H$, we have $f(H,P_t)\geq R_{t-2}(H)$ since a $(t-2)$-edge-colored graph trivially contains no rainbow $P_t$.
In particular, we have
\begin{equation}\label{eq:lower bound}
f(H,P_5)\geq R_3(H).
\end{equation}
Motivated by these results, we propose the following conjecture.

\begin{conjecture}\label{conj:P5}
For any graph $H$, we have $f(H,P_5)=R_3(H)$.
\end{conjecture}

In fact, an equivalent conjecture was first posed by Li, Besse, Magnant, Wang and Watts~\cite{LBMWW} in the language of Gallai--Ramsey numbers.
Given graphs $G$ and $H$ and a positive integer $k$, the {\it Gallai--Ramsey number} $gr_{k}(G : H)$ is the minimum integer $n$ such that every $k$-edge-coloring of $K_{n}$ contains either a rainbow copy of $G$ or a monochromatic copy of $H$.\footnote{This definition does not require all $k$ colors to be used. In the literature, a variant of the Gallai--Ramsey number requires that every color appears at least once (e.g., \cite{LiSi,LiWL,WMSZ}). In the case $G=P_5$, these two definitions yield distinct values; see \cite{LiWL} for a concrete example.}
In 2020, Li et al. conjectured that $gr_{k}(P_5 : H)= R_3(H)$ holds for any graph $H$ and $k\geq 3$ (see \cite[Conjecture~1]{LBMWW}).
Note that $f(H, P_5)=\max_{k\geq 1} gr_{k}(P_5 : H)$ and $gr_{k}(P_5 : H)\geq R_3(H)$ for $k\geq 3$.
Therefore, the assertion that $f(H,P_5)=R_3(H)$ is equivalent to the assertion that $gr_{k}(P_5 : H)= R_3(H)$ holds for $k\geq 3$.
Moreover, Li et al. obtained the following result (see \cite[Lemmas~1 and 2]{LBMWW}), generalizing the above-mentioned result of Gy\'{a}rf\'{a}s, Lehel and Schelp~\cite{GyLS}.

\begin{theorem}[\cite{LBMWW}]\label{th:LBMWW-1}
If $H$ is a connected graph or a bipartite graph, then $f(H,P_5)=R_3(H)$.
\end{theorem}

In this paper, we continue the study of $f(H, P_5)$.
By Theorem~\ref{th:LBMWW-1}, it suffices to study disconnected graphs with chromatic number at least 3.
For such graphs $H$, Conjecture~\ref{conj:P5} was only confirmed in very few cases.
In the subsequent statements, the notation $mG$ represents the disjoint union of $m$ copies of a graph $G$.
A graph is called {\it $r$-color-critical} if it has chromatic number $r$ and has an edge (called a {\it critical edge}) whose removal reduces the chromatic number to $r-1$.
We now list all disconnected graphs $H$ for which Conjecture~\ref{conj:P5} was confirmed:
\begin{itemize}
\item $H=mG$, where $m\geq 2$ and $G$ is a 3-color-critical graph; see \cite[Theorem~5]{LBMWW}.
\item $H=mK_r$, where $2\leq m\leq r-1$; see \cite[Theorem~6]{LBMWW}.
\item $H=2G$, where $G$ is a connected graph; see \cite[Theorem~7]{LBMWW}.
\end{itemize}
In the remainder of this paper, we will study this problem along the following three natural directions:
\begin{itemize}
\item[(1)] Given a graph $H$, find an integer $k$ as small as possible such that $f(H,P_5)\leq R_k(H)$.
\item[(2)] Let $H$ be the disjoint union of a connected graph and its subgraphs. Show that $f(H,P_5)= R_3(H)$.
\item[(3)] For graphs $H$ of chromatic number 3, show that $f(H,P_5)= R_3(H)$.
\end{itemize}
We will present our results for each of these three directions in the following three subsections, respectively.
We remark that all the listed results above are special cases of our newly established general theorems, as shown below.
In Section~\ref{sec:conclu}, we will also prove several results for a bipartite variant of the problem.
Before presenting our main results, we first introduce some additional terminology and notation.
\vspace{0.2cm}

\noindent{\bf Notation.}
For a positive integer $t$, let $[t]\colonequals \{1, 2, \ldots, t\}$.
A graph is called {\it nonempty} if it contains at least one edge.
For two graphs $G_1$ and $G_2$, we use $G_1\cup G_2$ to denote the disjoint union of $G_1$ and $G_2$, $G_1\vee G_2$ to denote the join of $G_1$ and $G_2$, and $G_1\subseteq G_2$ to denote that $G_1$ is a subgraph of $G_2$.
A {\it component} of a graph $G$ is a maximal connected subgraph of $G$.
For a graph $G$, let $\omega(G)$ be the {\it clique number} of $G$, that is, the maximum number of vertices in a complete subgraph of $G$.
We use $\overline{G}$ to denote the {\it complement} of a graph $G$.
When we refer to a (proper) vertex- or edge-coloring, we do not require that all colors be used.
Let $\chi(G)$ be the {\it chromatic number} of $G$, that is, the minimum number of colors needed in a proper vertex-coloring of $G$.
If $\chi(G)=k$, then $G$ is said to be {\it $k$-chromatic}.
Let $\sigma(G)$ be the {\it chromatic surplus} of $G$, defined as
$$\sigma(G)\colonequals \min\{|X|\colon\, \text{$X$ is a color class in a proper $\chi(G)$-vertex-coloring of $G$}\}.$$
Given an edge-colored graph $G$ and an edge $e\in E(G)$, let $c(e)$ be the color assigned to edge $e$.
For disjoint subsets $U,V\subseteq V(G)$, let $E(U, V)\colonequals\{uv\in E(G)\colon\, u\in U, v\in V\}$, $C(U, V)\colonequals \{c(e)\colon\, e\in E(U,V)\}$, and $C(U)\colonequals \{c(e)\colon\, e\in E(G[U])\}$.
If $\left|C(U, V)\right|=1$, then we use $c(U, V)$ to denote the unique color in $C(U, V)$.
The subgraph of $G$ induced by $U$ is denoted by $G[U]$, and $G-U$ is shorthand for $G[V(G)\setminus U]$.
If $U$ consists of a single vertex $u$, then we simply write $E(\{u\}, V)$, $C(\{u\}, V)$, $c(\{u\}, V)$ and $G-\{u\}$ as $E(u, V)$, $C(u, V)$, $c(u, V)$ and $G-u$, respectively.
If $G'\subseteq G$, we use $G-G'$ to denote $G-V(G')$.

\subsection{Integers $k$ with $f(H,P_5)\leq R_k(H)$}
\label{subsec:Rk}

Our first result is an upper bound on $f(H,P_5)$ in terms of the Ramsey number and chromatic number of $H$.
As a corollary, we have $f(H,P_5)= R_{3}(H)$ for every nonempty bipartite graph $H$.

\begin{theorem}\label{th:chro}
For any nonempty graph $H$, we have $f(H,P_5)\leq R_{\chi(H)+1}(H)$.
\end{theorem}

Let $G_1, G_2, \ldots, G_t$ be $t$ graphs of the same order.
We say that $G_1, G_2, \ldots, G_t$ are {\it $p$-homological} if $\max_{i\in [t]}\chi(G_i)=p$ and there exist positive integers $s_1, s_2, \ldots, s_p$ such that each $G_i$ admits a proper $p$-vertex-coloring whose color classes have sizes $s_1, s_2, \ldots, s_p$.
In other words, the graphs $G_1, G_2, \ldots, G_t$ are $p$-homological if $\max_{i\in [t]}\chi(G_i)=p$ and there exist positive integers $s_1, s_2, \ldots, s_p$ such that all of $G_1, G_2, \ldots, G_t$ are spanning subgraphs of $K_{s_1, \ldots, s_p}$.
We have the following result for the disjoint union of homological graphs.

\begin{theorem}\label{th:homology}
For $p\geq 2$, let $G_1, G_2, \ldots, G_t$ be connected $p$-homological graphs.
Then for $k=\max\{t,p,3\}$, we have $f(G_1\cup G_2\cup \cdots \cup G_t,P_5)\leq R_{k}(G_1\cup G_2\cup \cdots \cup G_t)$.
\end{theorem}

Note that when $p\geq t$, Theorem~\ref{th:homology} improves Theorem~\ref{th:chro} in the case $H=G_1\cup G_2\cup \cdots \cup G_t$.
Moreover, we have the following immediate consequence of Theorem~\ref{th:homology} and Inequality~(\ref{eq:lower bound}).

\begin{corollary}\label{cor:3G}
For any connected graph $G$ with $\chi(G)=3$, we have $f(3G,P_5)= R_{3}(3G)$.
\end{corollary}

\subsection{Disjoint unions of a connected graph and its subgraphs}
\label{subsec:subgraph}

Our first result in this direction concerns the disjoint union of a connected graph and one of its connected subgraphs.

\begin{theorem}\label{th:union-1}
Let $G_1$, $G_2$ be two connected graphs with $G_1\subseteq G_2$.
Then $f(G_1\cup G_2,P_5)= R_{3}(G_1\cup G_2)$.
\end{theorem}

As a corollary, we have $f(2G,P_5)= R_{3}(2G)$ for any connected graph $G$.
For more graphs, we have the following result.

\begin{theorem}\label{th:union-2}
Let $G, G_1, \ldots, G_t$ be connected graphs with $t\leq \frac{(\chi(G)-2)(R_2(G)-1)+\sigma(G)-1}{\chi(G)|V(G)|}$ and $G_1, \ldots, G_t\subseteq G$.
Then $f(G\cup G_1\cup \cdots \cup G_t,P_5)= R_{3}(G\cup G_1\cup \cdots \cup G_t)$.
\end{theorem}

Note that $R_2(K_r)\geq r^2+1$ for all $r\geq 4$.
Indeed, for $4\leq r\leq 11$, we have $R_2(K_r)\geq r^2+1$ from the known lower bounds on $R_2(K_r)$ (see \cite{Rad});
for $r\geq 12$, we have $R_2(K_r)\geq {r-1\choose 3}\geq r^2+1$ from Nagy's well-known construction~\cite{Nagy}.
Hence, $\frac{(\chi(K_r)-2)(R_2(K_r)-1)+\sigma(K_r)-1}{\chi(K_r)|V(K_r)|}\geq \frac{(r-2)r^2}{r^2}=r-2$ when $r\geq 4$.
Then by Theorem~\ref{th:union-2}, we have $f(mK_r,P_5)= R_{3}(mK_r)$ for $r\geq 4$ and $2\leq m\leq r-1$.
Moreover, we shall prove the following corollary of Theorem~\ref{th:union-2} in Section~\ref{sec:proof-subgraph}.

\begin{corollary}\label{cor:union-1}
Let $G, G_1, \ldots, G_t$ be connected graphs with $\chi(G)\geq \frac{t+4+\sqrt{(t+4)^2-12}}{2}$ and $G_1, \ldots, G_t\subseteq G$.
Then $f(G\cup G_1\cup \cdots \cup G_t,P_5)= R_{3}(G\cup G_1\cup \cdots \cup G_t)$.
\end{corollary}

\subsection{Graphs of chromatic number 3}
\label{subsec:3}

In this subsection, we focus on graphs of chromatic number 3.
The first result concerns disconnected graphs whose components are either 3-color-critical or bipartite.

\begin{theorem}\label{th:critical}
For any graph $H$ whose components are either 3-color-critical or bipartite, we have $f(H,P_5)= R_{3}(H)$.
\end{theorem}

As an immediate consequence, we have $f(mG,P_5)= R_{3}(mG)$, where $m\geq 2$ and $G$ is a 3-color-critical graph.
Moreover, since odd cycles are 3-color-critical graphs and even cycles are bipartite graphs, we have the following corollary.

\begin{corollary}\label{cor:cycles}
Let $H$ be any disjoint union of cycles. Then $f(H,P_5)= R_{3}(H)$.
\end{corollary}

Recall that $\sigma(G)$ is the chromatic surplus of a graph $G$.
Given a graph $G$ with $\chi(G)\leq 3$, let $\sigma_3(G)\colonequals \sigma(G)$ if $\chi(G)= 3$, and $\sigma_3(G)\colonequals 0$ otherwise.
In the following, we present two results related to $\sigma(G)$ or $\sigma_3(G)$.

\begin{theorem}\label{th:balanced}
For any $1\leq s\leq t$, let $G_1, \ldots, G_s$ be $s$ connected graphs of chromatic number 3 and $G_{s+1}, \ldots, G_t$ be $t-s$ connected bipartite graphs.
Suppose that for each $i\in [t]$, $G_i$ admits a proper $3$-vertex-coloring with color classes of sizes $a_i \geq b_i \geq c_i=\sigma_3(G_i)$ such that
\begin{itemize}
\item[{\rm (i)}] $a_i-b_i\leq 1$ for each $i\in [s]$, $a_i=b_i$ for each $i\in \{s+1, \ldots, t\}$, and
\item[{\rm (ii)}] $b_i+c_i\geq \sum_{j\in [t]\setminus \{i\}} c_j$ for each $i\in [t]$.
\end{itemize}
Then $f(G_1\cup \cdots \cup G_t,P_5)= R_{3}(G_1\cup \cdots \cup G_t)$.
\end{theorem}

\begin{theorem}\label{th:one3}
Let $H$ be a 3-chromatic graph satisfying
\begin{itemize}
\item[{\rm (i)}] $H$ has exactly one component of chromatic number 3, and
\item[{\rm (ii)}] every component of $H$ has order at least $\sigma(H)$.
\end{itemize}
Then $f(H,P_5)= R_{3}(H)$.
\end{theorem}

If $H$ is a graph with $\chi(H)=3$ and $\sigma(H)=1$, then $H$ has exactly one component of chromatic number 3.
Thus by Theorem~\ref{th:one3}, we have the following corollary.

\begin{corollary}\label{cor:sigma1}
Let $H$ be a graph with $\chi(H)=3$ and $\sigma(H)=1$. Then $f(H,P_5)= R_{3}(H)$.
\end{corollary}

The remainder of this paper is organized as follows.
In the next section, we provide some results that will be used in our proofs.
In Section~\ref{sec:proof-Rk}, we prove Theorems~\ref{th:chro} and \ref{th:homology}.
In Section~\ref{sec:proof-subgraph}, we present our proofs of Theorems~\ref{th:union-1}, \ref{th:union-2} and Corollary~\ref{cor:union-1}.
In Section~\ref{sec:proof-3}, we prove Theorems~\ref{th:critical}, \ref{th:balanced} and \ref{th:one3}.
Section~\ref{sec:conclu} contains concluding remarks, several open problems, and results for a bipartite variant.
In particular, we make some remarks to highlight the difficulty in solving Conjecture~\ref{conj:P5}.

\section{Preliminaries}
\label{sec:pre}

First, we point out that in order to study Conjecture~\ref{conj:P5}, it suffices to focus on graphs without isolated vertices.
This follows from repeated application of the following proposition.

\begin{proposition}\label{prop:isolated}
Let $H'$ be obtained from a graph $H$ by adding an isolated vertex.
If $f(H, P_5)=R_3(H)$, then $f(H', P_5)=R_3(H')$.
\end{proposition}

\begin{proof} By Inequality~(\ref{eq:lower bound}), it suffices to prove that $f(H', P_5)\leq R_3(H')$.
Note that $R_3(H')\geq R_3(H)$ and $R_3(H')\geq |V(H')|$.
Let $F$ be an edge-colored $K_{R_3(H')}$.
Since $R_3(H')\geq R_3(H)=f(H, P_5)$, $F$ contains either a rainbow $P_5$ or a monochromatic $H$.
If $F$ contains a monochromatic $H$, then since $R_3(H')\geq |V(H')|$ and $H'$ is obtained from $H$ by adding an isolated vertex, $F$ also contains a monochromatic $H'$.
Hence, $f(H', P_5)\leq R_3(H')$.
\end{proof}

In what follows, whenever we consider a monochromatic copy of a graph, we assume that it has no isolated vertices, unless otherwise stated.
In the remainder of this section, we introduce some Ramsey-type results and structural results that will be used in our proofs.

\subsection{Ramsey-type results}
\label{subsec:Ramsey}

For a family $\mathscr{H}$ of graphs, the Ramsey number $R_k(\mathscr{H})$ denotes the minimum integer $n$ such that every $k$-edge-colored $K_{n}$ contains a monochromatic copy of some graph in $\mathscr{H}$.
For a disconnected graph $H$, let $\mathscr{C}(H)$ be the set of all connected graphs containing $H$ as a subgraph (not necessarily spanning or induced).
Li, Besse, Magnant, Wang and Watts~\cite{LBMWW} obtained the following result linking $f(H, P_5)$, $R_3(H)$ and $R_2(\mathscr{C}(H))$.

\begin{lemma}[{\cite[Lemma~3]{LBMWW}}]\label{le:CH}
For any disconnected graph $H$, if $R_3(H)\geq R_2(\mathscr{C}(H))$, then $f(H, P_5)=R_3(H)$.
\end{lemma}

For two graphs $G_1$ and $G_2$, let $R(G_1, G_2)$ denote the minimum integer $n$ such that every $2$-edge-colored $K_{n}$ contains either a monochromatic $G_1$ of color~1 or a monochromatic $G_2$ of color~2.
For two families $\mathscr{G}$ and $\mathscr{H}$ of graphs, let $R(\mathscr{G},\mathscr{H})$ denote the minimum integer $n$ such that every $2$-edge-colored $K_{n}$ contains either a monochromatic copy of some graph in $\mathscr{G}$ of color~$1$ or a monochromatic copy of some graph in $\mathscr{H}$ of color~$2$.
If $\mathscr{G}=\{G\}$ or $\mathscr{H}=\{H\}$, we simply write $R(G,\mathscr{H})$ or $R(\mathscr{G},H)$, respectively.
We shall use the following results on 2-color Ramsey numbers.

\begin{lemma}\label{le:R2}
Let $m$ be a positive integer. The following statements hold.
\begin{itemize}
\item[{\rm (i)}]{\normalfont (\cite{Burr})} For any connected graph $H$, we have $R_2(H)\geq (\chi(H)-1)(|V(H)|-1)+\sigma(H)$.
\item[{\rm (ii)}]{\normalfont (\cite{Chv})} For any tree $T_{t+1}$ on $t+1$ vertices, we have $R(K_m, T_{t+1})=t(m-1)+1$.
\item[{\rm (iii)}]{\normalfont (\cite{FaSS})} For any graph $H$ with no isolated vertices, independence number $c$ and such that $K_c\cup K_{\left\lceil c/2 \right\rceil} \subseteq \overline{H}$, we have $R(mK_2, H)=\max\{|V(H)|+2m-c-1, |V(H)|+m-1\}$.
\end{itemize}
\end{lemma}

Applying Lemma~\ref{le:R2}~(iii), we can deduce the following corollary.

\begin{corollary}\label{cor:matching}
Let $H$ be a graph with no isolated vertices and with $c$ components, and let $m$ be a positive integer with $m\leq c$.
Then $R(mK_2, H)=|V(H)|+m-1$.
\end{corollary}

\begin{proof}
We first show that $R(mK_2, H)> |V(H)|+m-2$ by constructing a 2-edge-colored $K_{|V(H)|+m-2}$ as follows.
We partition the vertex set into two subsets $V_1$ and $V_2$ with $|V_1|=m-1$ and $|V_2|=|V(H)|-1$.
We color all edges within $V_1$ and all edges between $V_1$ and $V_2$ with color~1,
and all edges within $V_2$ with color~2.
Then there is no monochromatic $mK_2$ of color~1 or monochromatic $H$ of color~2.

We next show that $R(mK_2, H)\leq |V(H)|+m-1$.
Let $H_1, \ldots, H_c$ be the components of $H$.
Since $H$ contains no isolated vertices, we have $|V(H_i)|\geq 2$ for all $i\in [c]$.
Let $G=K_{|V(H_1)|} \cup \cdots \cup K_{|V(H_c)|}$.
Then $|V(H)|=|V(G)|$ and $R(mK_2, H)\leq R(mK_2, G)$, so it suffices to show that $R(mK_2, G)\leq |V(G)|+m-1$.
Note that the independence number of $G$ is $c$ and $2K_c\subseteq \overline{G}$.
Since $m\leq c$ and by Lemma~\ref{le:R2}~(iii), we have $R(mK_2, G)=\max\{|V(G)|+2m-c-1, |V(G)|+m-1\}= |V(G)|+m-1$.
The proof is complete.
\end{proof}

Given a graph $G$, a set $U\subseteq V(G)$ is an {\it odd cycle transversal} if $G-U$ is bipartite.
Let $\operatorname{oct}(G)$ denote the size of a minimum odd cycle transversal of $G$.
Note that for a bipartite graph $G$, we have $\operatorname{oct}(G)=0$, and for a $3$-chromatic graph $G$, we have $1\leq \operatorname{oct}(G)\leq \sigma(G)$.
We next provide several lower bounds on $3$-color Ramsey numbers.

\begin{lemma}\label{le:R3}
For connected nonempty graphs $G_1, \ldots, G_t$, the following statements hold.
\begin{itemize}
\item[{\rm (i)}] $R_3(G_1\cup \cdots \cup G_t)\geq \max_{1\leq i,j,\ell\leq t}\left\{(\chi(G_i)-1)(R(G_j, G_{\ell})-1)+\sum_{k\in [t], \chi(G_k)=\chi(G_i)}\sigma(G_k)\right\}.$
\item[{\rm (ii)}] $R_3(G_1\cup \cdots \cup G_t)\geq \max_{1\leq i,j\leq t}\left\{\left(R_2\left(K_{\omega(G_i)}\right)-1\right)(|V(G_j)|-1)+1\right\}.$
\item[{\rm (iii)}] If $\chi(G_i)=3$ for all $i\in [t]$, then $R_3(G_1\cup \cdots \cup G_t)\geq 3\sum_{i\in [t]}|V(G_i)|-2.$
\item[{\rm (iv)}] If $\max_{i\in [t]}\chi(G_i)= 3$, then $R_3(G_1\cup \cdots \cup G_t)\geq 2\sum_{i\in [t]}|V(G_i)|+\sum_{i\in [t]}\operatorname{oct}(G_i)-2.$
\end{itemize}
\end{lemma}

\begin{proof}
(i) We choose $i,j,\ell$ with $1\leq i,j,\ell \leq t$ arbitrarily.
Let $n=(\chi(G_i)-1)(R(G_j, G_{\ell})-1)+\sum_{k\in I_i}\sigma(G_k)-1$, where $I_i\colonequals \{k\in [t]\colon\, \chi(G_k)=\chi(G_i)\}$.
First, we consider the case $\sum_{k\in I_i}\sigma(G_k)\geq R(G_j, G_{\ell})+1$.
For each $k\in I_i$, we have $|V(G_k)|\geq \chi(G_i)\sigma(G_k)$.
Thus
\begin{align*}
~&~R_3(G_1\cup \cdots \cup G_t)\geq|V(G_1\cup \cdots \cup G_t)|=\sum\nolimits_{k=1}^{t}|V(G_k)|\geq \sum\nolimits_{k\in I_i}|V(G_k)| \geq \chi(G_i)\sum\nolimits_{k\in I_i}\sigma(G_k) \\
= &~(\chi(G_i)-1)\sum\nolimits_{k\in I_i}\sigma(G_k)+\sum\nolimits_{k\in I_i}\sigma(G_k) \geq (\chi(G_i)-1)(R(G_j, G_{\ell})+1)+\sum\nolimits_{k\in I_i}\sigma(G_k) > n,
\end{align*}
and we are done.
Next, we assume that $\sum_{k\in I_i}\sigma(G_k)\leq R(G_j, G_{\ell})$, so $\sum_{k\in I_i}\sigma(G_k)-1\leq R(G_j, G_{\ell})-1$.
Let $F$ be a 3-edge-colored $K_n$ defined as follows.
We partition $V(F)$ into $\chi(G_i)$ subsets $V_1, V_2, \ldots, V_{\chi(G_i)}$ such that $|V_m|=R(G_j, G_{\ell})-1$ for each $1\leq m\leq \chi(G_i)-1$, and $|V_{\chi(G_i)}|=\sum_{k\in I_i}\sigma(G_k)-1$.
For each $1\leq m\leq \chi(G_i)$, we color the edges within $V_m$ using colors 1 and 2 so that it contains neither a monochromatic $G_j$ of color~1 nor a monochromatic $G_{\ell}$ of color~2.
We color all the remaining edges using color~3.
Note that $F$ contains neither a monochromatic $G_j$ of color~1 nor a monochromatic $G_{\ell}$ of color~2.
Thus $F$ contains no monochromatic $G_1\cup \cdots \cup G_t$ of color~1 or 2.
Moreover, the edges of color~3 form a complete $\chi(G_i)$-partite graph.
For each $k\in I_i$, since $\chi(G_k)=\chi(G_i)$, every monochromatic $G_k$ of color~3 must contain at least $\sigma(G_k)$ vertices in each of $V_1, V_2, \ldots, V_{\chi(G_i)}$.
But $|V_{\chi(G_i)}|=\sum_{k\in I_i}\sigma(G_k)-1$, so $F$ contains no monochromatic $G_1\cup \cdots \cup G_t$ of color~3.
Therefore, $R_3(G_1\cup \cdots \cup G_t)\geq n+1=(\chi(G_i)-1)(R(G_j, G_{\ell})-1)+\sum_{k\in I_i}\sigma(G_k).$

(ii) We choose $i,j$ with $1\leq i,j \leq t$ arbitrarily.
Let $F$ be a 3-edge-colored $K_n$ defined as follows, where $n=\left(R_2\left(K_{\omega(G_i)}\right)-1\right)(|V(G_j)|-1)$.
We partition $V(F)$ into subsets $V_1, V_2, \ldots, V_{R_2\left(K_{\omega(G_i)}\right)-1}$ each of size $|V(G_j)|-1$.
By the definition, there exists a 2-edge-colored $K_{R_2\left(K_{\omega(G_i)}\right)-1}$ using colors 1 and 2 without monochromatic copies of $K_{\omega(G_i)}$.
Denote such a 2-edge-colored $K_{R_2\left(K_{\omega(G_i)}\right)-1}$ by $K$, and let $V(K)=\big\{v_1, v_2, \ldots, v_{R_2\left(K_{\omega(G_i)}\right)-1}\big\}$.
We color the edges between the $R_2\left(K_{\omega(G_i)}\right)-1$ parts of $F$ so that it forms a blow-up of $K$, that is, for each $1\leq m_1<m_2 \leq R_2\left(K_{\omega(G_i)}\right)-1$, we color all edges between $V_{m_1}$ and $V_{m_2}$ using the color of edge $v_{m_1}v_{m_2}$ in $K$.
We color all the remaining edges in $F$ using color~3.
Note that $F$ contains no monochromatic $G_j$ of color~3, so it contains no monochromatic $G_1\cup \cdots \cup G_t$ of color~3.
Moreover, $F$ contains no monochromatic $K_{\omega(G_i)}$ of color~1 or 2, so it contains no monochromatic $G_i$ (and thus no monochromatic $G_1\cup \cdots \cup G_t$) of color~1 or 2.
Therefore, $R_3(G_1\cup \cdots \cup G_t)\geq n+1 =\max_{1\leq i,j\leq t}\left\{\left(R_2\left(K_{\omega(G_i)}\right)-1\right)(|V(G_j)|-1)+1\right\}.$

(iii) Let $F$ be a 3-edge-colored $K_n$ defined as follows, where $n=3\sum_{i\in [t]}|V(G_i)|-3$.
We partition $V(F)$ into three subsets $V_1, V_2, V_3$ each of size $\sum_{i\in [t]}|V(G_i)|-1$.
We color all edges within $V_1$ and all edges between $V_2$ and $V_3$ with color~1,
all edges within $V_2$ and all edges between $V_3$ and $V_1$ with color~2,
and all edges within $V_3$ and all edges between $V_1$ and $V_2$ with color~3.
Since $G_i$ is connected and $\chi(G_i)=3$ for all $i\in [t]$, every monochromatic copy of $G_i$ of color $m$ ($m\in [3]$) in $F$ must be contained in $F[V_m]$.
Since $|V_1|=|V_2|=|V_3|=\sum_{i\in [t]}|V(G_i)|-1$, there is no monochromatic $G_1\cup \cdots \cup G_t$ in $F$.
Therefore, $R_3(G_1\cup \cdots \cup G_t)\geq n+1=3\sum_{i\in [t]}|V(G_i)|-2.$

(iv) Let $F$ be a 3-edge-colored $K_n$ defined as follows, where $n=2\sum_{i\in [t]}|V(G_i)|+\sum_{i\in [t]}\operatorname{oct}(G_i)-3$.
We partition $V(F)$ into three subsets $V_1, V_2, V_3$ with $|V_1|=|V_2|=\sum_{i\in [t]}|V(G_i)|-1$ and $|V_3|=\sum_{i\in [t]}\operatorname{oct}(G_i)-1$.
We color all edges within $V_1$ with color~1, all edges within $V_2$ with color~2, and all the remaining edges with color~3.
Note that $F$ contains no monochromatic $G_1\cup \cdots \cup G_t$ of color~1 or 2.
Suppose that there is a monochromatic $G_1\cup \cdots \cup G_t$ of color~3 in $F$.
For each $i\in [t]$, let $H_i$ be the corresponding copy of $G_i$ of color~3 in $F$, and let $U_i=V(H_i)\cap V_3$.
Since $H_i-U_i$ is bipartite, we have $|U_i|\geq \operatorname{oct}(G_i)$.
Since the copies $H_1, \ldots, H_t$ are vertex-disjoint, we have $|V_3|\geq \sum\nolimits_{i\in [t]}|U_i|\geq \sum\nolimits_{i\in [t]}\operatorname{oct}(G_i)$, contradicting $|V_3|=\sum_{i\in [t]}\operatorname{oct}(G_i)-1$.
Hence, $F$ contains no monochromatic $G_1\cup \cdots \cup G_t$ of color~3.
Therefore, $R_3(G_1\cup \cdots \cup G_t)\geq n+1 = 2\sum_{i\in [t]}|V(G_i)|+\sum_{i\in [t]}\operatorname{oct}(G_i)-2.$
\end{proof}

In the study of extremal graph theory problems, Simonovits~\cite{Sim983} introduced the decomposition family of graphs.
Given a graph $H$ with $\chi(H)=p\geq 3$, the {\it decomposition family} $\mathcal{M}(H)$ of $H$ is the set of graphs $M$ that are minimal under subgraph inclusion subject to the following property: there exists an integer $t$ such that $H\subseteq \left(M\cup \overline{K}_{t-|V(M)|}\right)\vee K_{(p-2)\times t},$ where $K_{(p-2)\times t}$ is the complete $(p-2)$-partite graph with $t$
vertices in each partite set.
In other words, a graph $M$ belongs to $\mathcal{M}(H)$ if the graph obtained from embedding $M$ (but not any of its proper subgraphs) into one partite set of a complete $(p-1)$-partite graph with large partite sets (e.g., each of size at least $|V(H)|$) contains $H$ as a subgraph.
We shall use the following Ramsey-type results related to the decomposition family of a graph.

\begin{lemma}\label{le:R-decom}
For two connected graphs $G$ and $H$ with $\chi(H)\geq 3$, the following statements hold.
\begin{itemize}
\item[{\rm (i)}] $R_2(H)\geq R(H, \mathcal{M}(H))+(\chi(H)-2)(|V(H)|-1)$.
\item[{\rm (ii)}] $R(G\cup H, \mathcal{M}(H))\leq \max\{R(G, H), R(H, \mathcal{M}(H))+|V(G)|\}$.
\end{itemize}
\end{lemma}

\begin{proof}
(i) Let $F$ be a 2-edge-colored $K_n$ defined as follows, where $n=R(H, \mathcal{M}(H))+(\chi(H)-2)(|V(H)|-1)-1$.
We partition $V(F)$ into $\chi(H)-1$ subsets $V_1, \ldots, V_{\chi(H)-1}$ with $|V_{\chi(H)-1}|=R(H, \mathcal{M}(H))-1$ and $|V_i|=|V(H)|-1$ for $i\in \{1, \ldots, \chi(H)-2\}$.
We color the edges within $V_{\chi(H)-1}$ so that it contains neither a monochromatic $H$ of color~1 nor a monochromatic graph in $\mathcal{M}(H)$ of color~2.
We color all edges within $V_i$ (for $i\in \{1, \ldots, \chi(H)-2\}$) using color~1, and all the remaining edges using color~2.
Since $H$ is connected, there is no monochromatic $H$ of color~1 in $F$.
Moreover, since $F[V_{\chi(H)-1}]$ contains no monochromatic graph in $\mathcal{M}(H)$ of color~2, there is no monochromatic $H$ of color~2 in $F$.

(ii) Let $n=\max\{R(G, H), R(H, \mathcal{M}(H))+|V(G)|\}$.
For a contradiction, suppose that $F$ is a $2$-edge-colored $K_n$ with no monochromatic $G\cup H$ of color~1 or monochromatic graph in $\mathcal{M}(H)$ of color~2.
By the minimality in the definition of $\mathcal{M}(H)$, we know that every graph $M\in \mathcal{M}(H)$ is a subgraph of $H$.
Then since $n\geq R(G, H)$, $F$ must contain a monochromatic copy $A$ of $G$ of color~1.
Let $F'=F-A$.
Then $|V(F')|=n-|V(G)|\geq R(H, \mathcal{M}(H))+|V(G)|-|V(G)|=R(H, \mathcal{M}(H))$.
Thus $F'$ contains a monochromatic copy $B$ of $H$ of color~1.
But then $A\cup B$ is a monochromatic $G\cup H$ of color~1, a contradiction.
\end{proof}

Given a bipartite graph $H$ and a positive integer $k$, the {\it $k$-color bipartite Ramsey number} $BR_k(H)$ is the minimum integer $n$ such that every $k$-edge-colored complete bipartite graph $K_{n,n}$ contains a monochromatic copy of $H$.
Note that a disconnected bipartite graph may admit several different bipartitions.
We define
\begin{align*}
  &~s(H)\colonequals \min \{|S|\colon\, \mbox{$S$ and $T$ are the partite sets of $H$ with $|S|\leq |T|$}\}, \\
  &~t(H)\colonequals \max \{|T|\colon\, \mbox{$S$ and $T$ are the partite sets of $H$ with $|S|\leq |T|$}\}, \\
  &~s^{\ast}(H)\colonequals \max \{|S|\colon\, \mbox{$S$ and $T$ are the partite sets of $H$ with $|S|\leq |T|$}\}, \\
  &~t^{\ast}(H)\colonequals \min \{|T|\colon\, \mbox{$S$ and $T$ are the partite sets of $H$ with $|S|\leq |T|$}\},
\end{align*}
where the minimum and maximum are taken over all bipartitions $(S,T)$ of $H$.
Then $|V(H)|=s(H)+t(H)=s^{\ast}(H)+t^{\ast}(H)$, $s(H)\leq s^{\ast}(H)\leq \frac{1}{2}|V(H)|\leq t^{\ast}(H)\leq t(H)$.
Moreover, if $H$ is connected, then $s(H)=s^{\ast}(H)$ and $t(H)=t^{\ast}(H)$.
We have the following lower bound on $BR_k(H)$ for connected bipartite graphs.
Note that this lower bound is best possible in general since $BR_k(K_{1,n})=k(n-1)+1$ (see \cite{HaHe}).

\begin{lemma}\label{le:bipar}
For any connected bipartite graph $H$, we have $BR_k(H)\geq k(t(H)-1)+1.$
\end{lemma}

\begin{proof}
It is well-known that the complete bipartite graph $K_{k,k}$ can be decomposed into $k$ edge-disjoint perfect matchings.
Thus we can obtain a $k$-edge-coloring $G$ of $K_{k,k}$ such that each color induces a perfect matching.
We construct a $k$-edge-coloring $F$ of $K_{k(t(H)-1),k(t(H)-1)}$ by substituting each vertex of $G$ with $t(H)-1$ new vertices, i.e., $F$ is a blow-up of $G$.
Then every color induces a $kK_{t(H)-1, t(H)-1}$ in $F$.
Since $H$ is connected and $K_{t(H)-1, t(H)-1}$ contains no $H$, there is no monochromatic copy of $H$ in $F$.
Hence, we have $BR_k(H)\geq k(t(H)-1)+1.$
\end{proof}

\subsection{Structural results}
\label{subsec:stru}

In 2007, Thomason and Wagner~\cite{ThWa} characterized the structure of edge-colored complete graphs with no rainbow $P_5$; their result is quite technical, and we refer the interested reader to \cite[Theorem~2]{ThWa} for more information.
Using the result of Thomason and Wagner, Li et al.~\cite{LBMWW} obtained the following technical lemma (see \cite[Lemma~1 and Remark~1]{LBMWW}).

\begin{lemma}[\cite{LBMWW}]\label{le:Caseb}
Let $H$ be a nonempty graph and $n\geq R_3(H)$.
Suppose that $K_{n}$ is edge-colored using exactly $k$ colors such that it contains neither a rainbow $P_5$ nor a monochromatic $H$.
Then we can partition the vertex set into $k-1$ parts $V_2, V_3, \ldots, V_k$ satisfying $($after possibly renumbering the colors$)$\footnote{We remark that there is no $V_1$ in this partition. The same holds in Lemmas~\ref{le:Caseb+} and \ref{le:LiWLP5}~(ii).}
\begin{itemize}
\item[{\rm (i)}] $\{i\}\subseteq C(V_i)\subseteq \{1,i\}$ for every $i\in \{2, 3, \ldots, k\}$, and
\item[{\rm (ii)}] $c(V_i, V_j)=1$ for every $2\leq i<j\leq k$.
\end{itemize}
\end{lemma}

Applying Lemma~\ref{le:Caseb}, we can further deduce the following result.

\begin{lemma}\label{le:Caseb+}
Let $H$ be a graph with nonempty components $H_1, \ldots, H_t$, and let $n\geq R_3(H)$.
Suppose that $K_{n}$ is edge-colored with no rainbow $P_5$ or monochromatic $H$.
Then we can partition the vertex set into $k-1$ parts $V_2, V_3, \ldots, V_k$ for some $k\geq 4$ satisfying $($after possibly renumbering the colors$)$
\begin{itemize}
\item[{\rm (i)}] $\{i\}\subseteq C(V_i)\subseteq \{1,i\}$ for every $i\in \{2, 3, \ldots, k\}$,
\item[{\rm (ii)}] $c(V_i, V_j)=1$ for every $2\leq i<j\leq k$, and
\item[{\rm (iii)}] $|V_2|\geq |V_3|\geq \cdots \geq |V_k|$, $|V_3|\geq \max_{i\in [t]}|V(H_i)|$, $|V_4|\geq \min_{i\in [t]}|V(H_i)|$ and $|V_3\cup \cdots \cup V_k|\geq |V(H)|$.
\end{itemize}
\end{lemma}

\begin{proof} Let $F$ be an edge-colored $K_{n}$ using exactly $k$ colors such that it contains neither a rainbow $P_5$ nor a monochromatic $H$.
Then $k\geq 4$; otherwise $F$ contains a monochromatic copy of $H$ since $n\geq R_3(H)$.
By Lemma~\ref{le:Caseb}, we can partition $V(F)$ into $k-1$ parts $V_2, V_3, \ldots, V_{k}$ satisfying (i) and (ii).
Without loss of generality, we may assume that $|V_2|\geq |V_3|\geq \cdots \geq |V_{k}|$.
Let $F'$ be a $3$-edge-coloring of $K_n$ obtained from $F$ by recoloring all edges of colors in $\{4, \ldots, k\}$ with color $3$.
Since $n\geq R_{3}(H)$, there is a monochromatic copy of $H$ in $F'$.
Denote such a monochromatic $H$ by $A$.
Since every component of $A$ is connected and every edge between two distinct
parts $V_i$ and $V_j$ has color~$1$, each component of $A$ is contained
entirely in one of $V_3, V_4, \ldots, V_k$. Moreover, the components of $A$ cannot
all be contained in the same part, since otherwise $F$ would contain a
monochromatic copy of $H$. Hence, at least two of $V_3, V_4, \ldots, V_k$
contain components of $A$. Since
$|V_2|\ge |V_3|\ge\cdots\ge |V_k|,$
it follows that
$|V_3|\ge \max_{i\in[t]}|V(H_i)|$ and
$|V_4|\ge \min_{i\in[t]}|V(H_i)|.$
Finally, since all vertices of $A$ are contained in
$V_3\cup \cdots \cup V_k$, we have
$|V_3\cup \cdots \cup V_k|\ge |V(H)|.$
\end{proof}

In 2019, Li, Wang and Liu~\cite{LiWL} characterized the structures of edge-colored complete bipartite graphs without rainbow $P_4$ or $P_5$.\footnote{The original statements of Li, Wang and Liu are formulated somewhat differently (see \cite[Theorems~1.4 and 1.5]{LiWL}). Lemmas~\ref{le:LiWLP4} and \ref{le:LiWLP5} serve as convenient reformulations of their results.}

\begin{lemma}[\cite{LiWL}]\label{le:LiWLP4}
Let $K_{n,n}$ $(n \geq 2)$ be edge-colored with exactly $k$ $(k\geq 3)$ colors such that it contains no rainbow $P_4$.
Let $U$ and $V$ be the partite sets.
Then $U$ can be partitioned into $k$ parts $U_1, U_2, \ldots, U_k$ such that $c(U_i, V)=i$ for every $i\in [k]$ $($after possibly renumbering the colors and interchanging $U$ and $V$$)$.
\end{lemma}

\begin{lemma}[\cite{LiWL}]\label{le:LiWLP5}
Let $K_{n,n}$ $(n \geq 3)$ be edge-colored with exactly $k$ $(k\geq 4)$ colors such that it contains no rainbow $P_5$.
Let $U$ and $V$ be the partite sets.
Then one of the following holds $($after possibly renumbering the colors and interchanging $U$ and $V$$)$:
\begin{itemize}
\item[{\rm (i)}] $U$ can be partitioned into two parts $U_1, U_2$ with $|U_1|\geq 1$ and $|U_2|\geq 0$, and $V$ can be partitioned into $k$ parts $V_1, V_2, \ldots, V_k$ with $|V_1|\geq 0$ and $|V_i|\geq 1$ for every $i\in \{2, 3, \ldots, k\}$, such that $c(U_2, V)=1$ and $c(U_1, V_i)=i$ for every $i\in [k]$;

\item[{\rm (ii)}] $U$ can be partitioned into $k-1$ parts $U_2, U_3, \ldots, U_k$, and $V$ can be partitioned into $k-1$ parts $V_2, V_3, \ldots, V_k$, such that $\{i\}\subseteq C(U_i, V_i)\subseteq \{1,i\}$ for every $i\in \{2, 3, \ldots, k\}$, and $c(U_i, V_j)=1$ for every $2\leq i\neq j\leq k$;

\item[{\rm (iii)}] $n\in \{3,4\}$, $k=4$, and each color induces a matching.
\end{itemize}
\end{lemma}

\section{Proofs of results related to $f(H,P_5)\leq R_k(H)$}
\label{sec:proof-Rk}

In this section, we present our proofs of Theorems~\ref{th:chro} and \ref{th:homology}.
Our proofs extend a recoloring argument of Li et al.~\cite{LBMWW}.

\vspace{0.3cm}
\noindent{\bf Proof of Theorem~\ref{th:chro}.}~
By repeatedly deleting isolated vertices, it suffices to prove the theorem when $H$ contains no isolated vertices.
Indeed, if $H'$ is obtained from $H$ by adding an isolated vertex, then $R_{\chi(H)+1}(H')\geq R_{\chi(H)+1}(H)$ and $R_{\chi(H)+1}(H')\geq |V(H')|$,
so a monochromatic copy of $H$ can be extended to a monochromatic copy of $H'$ by adding an unused vertex.

For a contradiction, suppose that $F$ is an edge-colored $K_n$ without rainbow $P_5$ or monochromatic $H$, where $n=R_{\chi(H)+1}(H)$.
After renumbering the colors, assume that the set of colors used by $F$ is $[k]$.
Then $k\geq \chi(H)+2$, since otherwise $F$ contains a monochromatic copy of $H$.
By Lemma~\ref{le:Caseb}, we can partition $V(F)$ into $k-1$ parts $V_2, V_3, \ldots, V_k$ satisfying
\begin{itemize}
\item[{\rm (i)}] $\{i\}\subseteq C(V_i)\subseteq \{1,i\}$ for every $i\in \{2, 3, \ldots, k\}$, and
\item[{\rm (ii)}] $c(V_i, V_j)=1$ for every $2\leq i<j\leq k$.
\end{itemize}

Choose a partition $I_1, \ldots, I_{\chi(H)}$ of $\{2, 3, \ldots, k\}$ and correspondingly define sets $A_1=\bigcup_{i\in I_1}V_i, \ldots, A_{\chi(H)}=\bigcup_{i\in I_{\chi(H)}}V_i$ such that the following conditions hold:
\begin{itemize}
\item[] (1) ~$|A_1|\geq \cdots \geq |A_{\chi(H)}|$;
\item[] (2) ~$|A_{\chi(H)}|$ is largest subject to (1);
\item[] (3) ~$|A_{\chi(H)-1}|$ is largest subject to (1) and (2);
\item[] ~~$\vdots$
\item[] ($\chi(H)$) ~$|A_{2}|$ is largest subject to $(1), (2), \ldots,(\chi(H)-1)$.
\end{itemize}

\begin{claim}\label{cl:chro-1}
If $|I_{\ell}|\geq 2$ for some $1\leq \ell \leq \chi(H)$, then $|A_{\chi(H)}|\geq \frac{1}{2}|A_{\ell}|$.
\end{claim}

\begin{proof}
Suppose for a contradiction that there exists an $\ell$ with $|I_{\ell}|\geq 2$ and $|A_{\chi(H)}|< \frac{1}{2}|A_{\ell}|$.
Let $i_0$ be the smallest index with $|A_{i_0}|=|A_{i_0+1}|=\cdots =|A_{\chi(H)}|$.
Then $\ell < i_0 \leq \chi(H)$, $|A_{i_0-1}|>|A_{i_0}|$ and $|A_{i_0}|< \frac{1}{2}|A_{\ell}|$.
Since $|I_{\ell}|\geq 2$, there exists a $j\in I_{\ell}$ with $|V_j|\leq \frac{1}{2}|A_{\ell}|$.
Let
$$I'_i=
\left\{
   \begin{aligned}
    &I_i & & \mbox{for $i\in \{1,2, \ldots, \chi(H)\}\setminus \{\ell,i_0\}$},\\
    &I_i\setminus \{j\} & & \mbox{for $i=\ell$},\\
    &I_i\cup \{j\} & & \mbox{for $i=i_0$}.
   \end{aligned}
   \right.$$
Let $\varphi$ be a permutation of $1,2, \ldots, \chi(H)$ such that $\big|\bigcup_{i\in I'_{\varphi(1)}}V_i\big|\geq \cdots \geq \big|\bigcup_{i\in I'_{\varphi(\chi(H))}}V_i\big|.$

If $|A_{\ell}\setminus V_j|\leq |A_{i_0}\cup V_j|$, then since $|A_{i_0}|< \frac{1}{2}|A_{\ell}|$ and $|V_j|\leq \frac{1}{2}|A_{\ell}|$, we have
$$|A_{i_0}\cup V_j|\geq |A_{\ell}\setminus V_j|= |A_{\ell}|-|V_j|\geq |A_{\ell}|-\frac{1}{2}|A_{\ell}|=\frac{1}{2}|A_{\ell}| > |A_{i_0}|.$$
If $|A_{\ell}\setminus V_j|> |A_{i_0}\cup V_j|$, then $|A_{\ell}\setminus V_j|> |A_{i_0}\cup V_j|> |A_{i_0}|.$
In both cases, we have $\big|\bigcup_{i\in I'_{\varphi(i_0)}}V_i\big|>|A_{i_0}|$, and $|A_{i_0+1}|=|A_{i_0+2}|=\cdots =|A_{\chi(H)}|=\big|\bigcup_{i\in I'_{\varphi(\chi(H))}}V_i\big|=\big|\bigcup_{i\in I'_{\varphi(\chi(H)-1)}}V_i\big|=\cdots=\big|\bigcup_{i\in I'_{\varphi(i_0+1)}}V_i\big|$.
Thus after reordering, the sizes of sets in positions $i_0+1, \ldots, \chi(H)$ remain unchanged, while the set in position $i_0$ becomes strictly larger.
This contradicts the choice of the partition $I_1, \ldots, I_{\chi(H)}$ of $\{2, 3, \ldots, k\}$.
\end{proof}

Let $F'$ be a $(\chi(H)+1)$-edge-coloring of $K_n$ obtained from $F$ by recoloring all edges of colors in $I_i$ with color $i+1$ for every $1\leq i\leq \chi(H)$.
Since $n=R_{\chi(H)+1}(H)$, there is a monochromatic copy of $H$ in $F'$, say of color $m$.
Since $F$ contains no monochromatic copy of $H$, we have $m \geq 2$ and $|I_{m-1}|\geq 2$.
By Claim~\ref{cl:chro-1}, we have $|A_{\chi(H)}|\geq \frac{1}{2}|A_{m-1}|$.
Hence, $|A_1|\geq \cdots \geq |A_{m-1}|\geq |V(H)|$ and $|A_{m}|\geq \cdots \geq |A_{\chi(H)}|\geq \frac{1}{2}|A_{m-1}|\geq \frac{1}{2}|V(H)|$.
Let $s_1, s_2, \ldots, s_{\chi(H)}$ be the sizes of color classes in a proper $\chi(H)$-vertex-coloring of $H$, and assume that $s_1\geq s_2\geq \cdots\geq s_{\chi(H)}$.
Then $|A_1|\geq |V(H)| \geq s_1$ and $|A_i|\geq \frac{1}{2}|V(H)|\geq s_2$ for each $2\leq i\leq \chi(H)$.
Thus the edges between the parts $A_1, A_2, \ldots, A_{\chi(H)}$ in $F$ form a monochromatic graph of color~1 that contains $K_{s_1, s_2, \ldots, s_2}$ (and thus $H$) as a subgraph, a contradiction.
This completes the proof of Theorem~\ref{th:chro}.
\hfill$\square$
\vspace{0.2cm}

\vspace{0.3cm}
\noindent{\bf Proof of Theorem~\ref{th:homology}.}~
Let $H=G_1\cup G_2\cup \cdots \cup G_t$.
Since $G_1, G_2, \ldots, G_t$ are $p$-homological, we may assume that
there exist positive integers $s_1, s_2, \ldots, s_p$ such that all of $G_1, G_2, \ldots, G_t$ are spanning subgraphs of $K_{s_1, \ldots, s_p}$.
Let $s=s_1+s_2+\cdots+s_p$, so $|V(G_1)|=|V(G_2)|=\cdots=|V(G_t)|=s$.

For a contradiction, suppose that $F$ is an edge-colored $K_n$ without rainbow $P_5$ or monochromatic $H$, where $n=R_{k}(H)$ and $k=\max\{t,p,3\}$.
After relabeling the colors, assume the set of colors used on $E(F)$ is $[k']$.
Then $k'\geq k+1=\max\{t,p,3\}+1$, since otherwise $F$ contains a monochromatic copy of $H$.
By Lemma~\ref{le:Caseb}, we can partition $V(F)$ into $k'-1$ parts $V_2, V_3, \ldots, V_{k'}$ satisfying
\begin{itemize}
\item[{\rm (i)}] $\{i\}\subseteq C(V_i)\subseteq \{1,i\}$ for every $i\in \{2, 3, \ldots, k'\}$, and
\item[{\rm (ii)}] $c(V_i, V_j)=1$ for every $2\leq i<j\leq k'$.
\end{itemize}
Without loss of generality, we may assume that $|V_2|\geq |V_3|\geq \cdots \geq |V_{k'}|$.

Let $F'$ be a $k$-edge-coloring of $K_n$ obtained from $F$ by recoloring all edges of colors in $\{k+1, \ldots, k'\}$ with color $k$.
Since $n=R_{k}(H)$, there is a monochromatic copy of $H$ in $F'$.
Denote such a monochromatic $H$ by $A$.
Since $F$ contains no monochromatic copy of $H$, the color of $A$ must be $k$.
Moreover, since every $G_i$ is connected, each component of $A$ lies entirely in one of $V_k, V_{k+1}, \ldots, V_{k'}$.
If all components lay in the same $V_j$, then $F$ contains a monochromatic $H$ of color $j$, a contradiction.
Hence, the components of $A$ must appear in at least two distinct parts from $V_k, V_{k+1}, \ldots, V_{k'}$.
Let $i_0$ be the largest index such that $A$ has a component within $V_{i_0}$.
Then $i_0\geq k+1$ and $|V_2|\geq |V_3|\geq \cdots \geq |V_{k+1}|\geq |V_{i_0}|\geq s.$

For each $i\in\{2,\ldots,k+1\}$, let $V_{i,1},\ldots,V_{i,p}$ be pairwise disjoint subsets of $V_i$ with $|V_{i,j}|=s_j$ for each $j\in[p]$.
Let\footnote{For clarity, we remark here that when $k=p$, only $W_{1}, W_2, \ldots, W_p(=W_k)$ exist, while all other $W_i$ do not exist.}
\begin{align*}
  &~W_1=V_{2,1}\cup V_{3,2}\cup \cdots \cup V_{p+1,p}, \\
  &~W_2=V_{2,2}\cup V_{3,3}\cup \cdots \cup V_{p,p}\cup V_{k+1,1}, \\
  &~W_3=V_{2,3}\cup V_{3,4}\cup \cdots \cup V_{p-1,p}\cup V_{k,1}\cup V_{k+1,2}, \\
  &~~~\vdots ~ \\
  &~W_p=V_{2,p}\cup V_{k-p+3,1}\cup V_{k-p+4,2}\cup \cdots \cup V_{k+1,p-1}, \\
  &~W_{p+1}=V_{k-p+2,1}\cup V_{k-p+3,2}\cup \cdots \cup V_{k+1,p}, \\
  &~W_{p+2}=V_{k-p+1,1}\cup V_{k-p+2,2}\cup \cdots \cup V_{k,p}, \\
  &~~~\vdots ~ \\
  &~W_{k}=V_{3,1}\cup V_{4,2}\cup \cdots \cup V_{p+2,p}.
\end{align*}
Then $W_1, W_2, \ldots, W_k$ are disjoint vertex subsets.
Moreover, for each $i\in [k]$, $F[W_i]$ contains a monochromatic $K_{s_1, \ldots, s_p}$ of color~1.
Since $k\geq t$, $F$ contains a monochromatic $tK_{s_1, \ldots, s_p}$ (and thus a monochromatic $G_1\cup G_2\cup \cdots \cup G_t$).
This contradiction completes the proof of Theorem~\ref{th:homology}.
\hfill$\square$

\section{Proofs of results for the disjoint unions of a connected graph and its subgraphs}
\label{sec:proof-subgraph}

In this section, we present our proofs of Theorems~\ref{th:union-1}, \ref{th:union-2} and Corollary~\ref{cor:union-1}.
We first prove Theorem~\ref{th:union-2} and Corollary~\ref{cor:union-1}.
In our proof of Theorem~\ref{th:union-2}, we shall apply Lemma~\ref{le:CH}.
To this end, we first prove the following result on $R_2(\mathscr{C}(G\cup G_1\cup \cdots \cup G_t))$.

\begin{lemma}\label{le:union-2}
Let $G, G_1, \ldots, G_t$ be connected graphs with $t\leq \frac{(\chi(G)-2)(R_2(G)-1)+\sigma(G)-1}{\chi(G)|V(G)|}$ and $G_1, \ldots, G_t\subseteq G$.
Then $R_2(\mathscr{C}(G\cup G_1\cup \cdots \cup G_t))\leq (\chi(G)-1)(R_2(G)-1)+\sigma(G).$
\end{lemma}

\begin{proof} For a contradiction, suppose that $F$ is a $2$-edge-colored $K_n$ without monochromatic copies of any graph in $\mathscr{C}(G\cup G_1\cup \cdots \cup G_t)$, where $n=(\chi(G)-1)(R_2(G)-1)+\sigma(G)$.
A classical observation of Erd\H{o}s and Rado states that every $2$-edge-colored complete graph contains a monochromatic spanning tree.
Hence, without loss of generality, we may assume that $F$ contains a monochromatic spanning tree of color~1.
Then $F$ contains no monochromatic $G\cup G_1\cup \cdots \cup G_t$ (and thus no monochromatic $(t+1)G$) of color~1; otherwise $F$ contains a monochromatic copy of some graph in $\mathscr{C}(G\cup G_1\cup \cdots \cup G_t)$.
We choose a set $\{A_1, \ldots, A_k\}$ of disjoint monochromatic copies of $G$ of color~1 in $F$ such that $k$ is maximal.
Then $k\leq t$.
Let $F'=F-(A_1\cup \cdots \cup A_k)$.
Then $F'$ contains no monochromatic $G$ of color~1.

Since $t\leq \frac{(\chi(G)-2)(R_2(G)-1)+\sigma(G)-1}{\chi(G)|V(G)|}$, we have
\begin{align*}
  |V(F')|= &~n-k|V(G)| ~\geq n-t|V(G)| = (\chi(G)-1)(R_2(G)-1)+\sigma(G)-t|V(G)| \\
  = &~R_2(G)+(\chi(G)-2)(R_2(G)-1)+\sigma(G)-1-t|V(G)| \\
  \geq &~R_2(G)+t\chi(G)|V(G)|-t|V(G)| = R_2(G)+t(\chi(G)-1)|V(G)|.
\end{align*}
Repeatedly find a monochromatic copy of $G$ and delete its vertices.
After $t(\chi(G)-1)$ such copies have been removed, at least $R_2(G)$ vertices remain, and hence the remaining 2-edge-colored complete graph contains another monochromatic copy of $G$.
Hence, $F'$ contains $t(\chi(G)-1)+1$ vertex-disjoint monochromatic copies $B_1, \ldots, B_{t(\chi(G)-1)+1}$ of $G$.
Since $F'$ contains no monochromatic $G$ of color~1, all of $B_1, \ldots, B_{t(\chi(G)-1)+1}$ are of color~2.

We define an auxiliary $2$-edge-colored complete graph $\Gamma$ as follows.
Let $V(\Gamma)=\{v_1, \ldots, v_{t(\chi(G)-1)+1}\}$.
For $1\leq i< j \leq t(\chi(G)-1)+1$, the edge $v_iv_j$ is of color~1 if all edges between $V(B_i)$ and $V(B_j)$ are of color~1 in $F'$, and of color~2 if there exists an edge of color~2 between $V(B_i)$ and $V(B_j)$ in $F'$.
By Lemma~\ref{le:R2}~(ii), $\Gamma$ contains either a monochromatic $K_{\chi(G)}$ of color~1 or a monochromatic tree $T_{t+1}$ on $t+1$ vertices of color~2.
If $\Gamma$ contains a monochromatic $K_{\chi(G)}$ of color~1, say with vertex set $\{v_1, \ldots, v_{\chi(G)}\}$, then every edge between $B_1, \ldots, B_{\chi(G)}$ is color~1.
This implies that $F'$ contains a monochromatic $G$ of color~1, a contradiction.
If $\Gamma$ contains a monochromatic $T_{t+1}$ of color~2, then $F'$ contains a monochromatic connected subgraph of color~2 that contains $(t+1)G$ (and thus $G\cup G_1\cup \cdots \cup G_t$) as a subgraph, a contradiction.
This completes the proof of Lemma~\ref{le:union-2}.
\end{proof}

We are now ready to prove Theorem~\ref{th:union-2} and Corollary~\ref{cor:union-1}.

\vspace{0.3cm}
\noindent{\bf Proof of Theorem~\ref{th:union-2}.}~
By Lemma~\ref{le:union-2}, we have $R_2(\mathscr{C}(G\cup G_1\cup \cdots \cup G_t))\leq (\chi(G)-1)(R_2(G)-1)+\sigma(G).$
By Lemma~\ref{le:R3}~(i), we have $R_3(G\cup G_1\cup \cdots \cup G_t)\geq (\chi(G)-1)(R_2(G)-1)+\sigma(G).$
Thus $R_3(G\cup G_1\cup \cdots \cup G_t)\geq R_2(\mathscr{C}(G\cup G_1\cup \cdots \cup G_t)).$
By Lemma~\ref{le:CH}, we have $f(G\cup G_1\cup \cdots \cup G_t, P_5)=R_3(G\cup G_1\cup \cdots \cup G_t)$.
This completes the proof of Theorem~\ref{th:union-2}.
\hfill$\square$
\vspace{0.2cm}

\vspace{0.3cm}
\noindent{\bf Proof of Corollary~\ref{cor:union-1}.}~
By Lemma~\ref{le:R2}~(i), we have $R_2(G)\geq (\chi(G)-1)(|V(G)|-1)+\sigma(G)$.
Thus
\begin{align}\label{eq:Cor-union-1}
  &~(\chi(G)-2)(R_2(G)-1)+\sigma(G)-1-t\chi(G)|V(G)| \nonumber\\
  \geq &~(\chi(G)-2)((\chi(G)-1)(|V(G)|-1)+\sigma(G)-1)+\sigma(G)-1-t\chi(G)|V(G)| \nonumber\\
  = &~((\chi(G)-2)(\chi(G)-1)-t\chi(G))|V(G)|-(\chi(G)-2)(\chi(G)-1)+(\chi(G)-1)(\sigma(G)-1) \nonumber\\
  = &~(\chi(G)^2-(t+3)\chi(G)+2)|V(G)|+(\chi(G)-1)(\sigma(G)-\chi(G)+1) \nonumber\\
  = &~(\chi(G)^2-(t+4)\chi(G)+3)|V(G)|+(\chi(G)-1)(|V(G)|+\sigma(G)-\chi(G)+1) \nonumber\\
  \geq &~(\chi(G)^2-(t+4)\chi(G)+3)|V(G)|,
\end{align}
where the last inequality follows from $|V(G)|\geq \chi(G)$ and $\sigma(G)\geq 1$.
Since $\chi(G)\geq \frac{t+4+\sqrt{(t+4)^2-12}}{2}$, we have $\chi(G)^2-(t+4)\chi(G)+3\geq 0$.
Combining with Inequality~(\ref{eq:Cor-union-1}), we have $t\leq \frac{(\chi(G)-2)(R_2(G)-1)+\sigma(G)-1}{\chi(G)|V(G)|}$.
By Theorem~\ref{th:union-2}, the result follows.
\hfill$\square$
\vspace{0.2cm}

Next, we prove Theorem~\ref{th:union-1}.
We first prove the following result on $R_2(\mathscr{C}(G_1\cup G_2))$.

\begin{lemma}\label{le:union-1}
Let $G_1$, $G_2$ be two connected graphs with $G_1\subseteq G_2$ and $\chi(G_2)\geq 3$. Then the following statements hold.
\begin{itemize}
\item[{\rm (i)}] $R_2(\mathscr{C}(G_1\cup G_2))\leq R_2(G_2)+\chi(G_2)|V(G_2)|$.
\item[{\rm (ii)}] If $\chi(G_1)<\chi(G_2)$, then $R_2(\mathscr{C}(G_1\cup G_2))\leq R_2(G_2)+(\chi(G_2)-1)|V(G_2)|$.
\end{itemize}
\end{lemma}

\begin{proof}
(i) Let $n=R_2(G_2)+\chi(G_2)|V(G_2)|$.
For a contradiction, suppose that $F$ is a $2$-edge-colored $K_n$ containing no monochromatic copy of any graph in $\mathscr{C}(G_1\cup G_2)$.
Since $G_1\subseteq G_2$, $F$ contains no monochromatic copy of any graph in $\mathscr{C}(2G_2)$ either.
Since $n=R_2(G_2)+\chi(G_2)|V(G_2)|$, $F$ contains $\chi(G_2)+1$ disjoint monochromatic copies $A_1, \ldots, A_{\chi(G_2)+1}$ of $G_2$.

\begin{claim}\label{cl:union-1-1}
Among the graphs $A_1, \ldots, A_{\chi(G_2)+1}$, at least $\chi(G_2)$ of them are of the same color.
\end{claim}

\begin{proof}
For a contradiction, suppose that for each color $i\in [2]$, at most $\chi(G_2)-1$ of $A_1, \ldots, A_{\chi(G_2)+1}$ are of color $i$.
Since $\chi(G_2)+1\geq 4$, there exist four distinct indices $a,b,c,d$ such that $A_a$ and $A_b$ are of color~1, and $A_c$ and $A_d$ are of color~2.
Since every $2$-edge-colored complete graph contains a monochromatic spanning tree, either $A_a$ and $A_b$ are connected by a path of color~1, or $A_c$ and $A_d$ are connected by a path of color~2.
This implies that $F$ contains a connected monochromatic graph containing $2G_2$ as a subgraph, a contradiction.
\end{proof}

By Claim~\ref{cl:union-1-1}, we may assume that $A_1, \ldots, A_{\chi(G_2)}$ are of color~1 without loss of generality.
To avoid a connected monochromatic graph of color~1 containing $2G_2$, we have $c(V(A_i), V(A_j))=2$ for every $1\leq i<j\leq \chi(G_2)$.
Let $s_1, \ldots, s_{\chi(G_2)}$ be the sizes of color classes in a proper $\chi(G_2)$-vertex-coloring of $G_2$.
For each $i\in [\chi(G_2)]$, we can partition $V(A_i)$ into subsets $V_{i,1}, \ldots, V_{i,\chi(G_2)}$ with $|V_{i,j}|=s_j$ for each $j\in [\chi(G_2)]$.
Then each of $F\left[V_{1,1}\cup V_{2,2}\cup \cdots \cup V_{\chi(G_2), \chi(G_2)}\right]$ and $F\left[V_{1,2}\cup V_{2,3}\cup \cdots \cup V_{\chi(G_2)-1, \chi(G_2)}\cup V_{\chi(G_2),1}\right]$ contains a monochromatic $G_2$ of color~2.
Since $c(V(A_i), V(A_j))=2$ for every $1\leq i<j\leq \chi(G_2)$, $F$ contains a connected monochromatic graph of color~2 containing $2G_2$ as a subgraph, a contradiction.

(ii) The proof of (ii) is analogous to that of (i), so we only present the different parts.
Suppose for a contradiction that $F$ is a $2$-edge-colored $K_n$ containing no monochromatic copy of any graph in $\mathscr{C}(G_1\cup G_2)$, where $n=R_2(G_2)+(\chi(G_2)-1)|V(G_2)|$.
Then $F$ contains $\chi(G_2)$ disjoint monochromatic copies $A_1, \ldots, A_{\chi(G_2)}$ of $G_2$.
If $\chi(G_2)=3$, then at least two of $A_1, A_2, A_3$ are of the same color by the pigeonhole principle;
if $\chi(G_2)\geq 4$, then by an analogous argument to the proof of Claim~\ref{cl:union-1-1}, we can also deduce that at least $\chi(G_2)-1$ of $A_1, \ldots, A_{\chi(G_2)}$ are of the same color.
We next show that all the $\chi(G_2)$ graphs are of the same color.

\begin{claim}\label{cl:union-1-2}
All of $A_1, \ldots, A_{\chi(G_2)}$ are of the same color.
\end{claim}

\begin{proof}
From the above argument, we know that at least $\chi(G_2)-1$ of $A_1, \ldots, A_{\chi(G_2)}$ are of the same color.
Suppose the claim does not hold.
Then without loss of generality, we may assume that $A_1, \ldots, A_{\chi(G_2)-1}$ are of color~1 and $A_{\chi(G_2)}$ is of color~2.
To avoid a connected monochromatic graph of color~1 containing $2G_2$ as a subgraph, we have $c(V(A_i), V(A_j))=2$ for every $1\leq i<j\leq \chi(G_2)-1$.
If all edges between $V(A_{\chi(G_2)})$ and $V(A_1)\cup \cdots \cup V(A_{\chi(G_2)-1})$ are of color~1, then two of $A_1, \ldots, A_{\chi(G_2)-1}$ are connected by a path of color~1, and thus form a monochromatic graph of color~1 containing $2G_2$ as a subgraph, a contradiction.
Hence, there exists an edge of color~2 between $V(A_{\chi(G_2)})$ and $V(A_1)\cup \cdots \cup V(A_{\chi(G_2)-1})$, say $c(uv)=2$ with $u\in V(A_{\chi(G_2)})$ and $v\in V(A_1)$.
Moreover, since $G_1\subseteq G_2$ and $\chi(G_1)<\chi(G_2)$, there is a monochromatic $G_1$ of color~2 in $F\left[V(A_1)\cup \cdots \cup V(A_{\chi(G_2)-1})\right]$.
Using the edge $uv$ and edges between $A_1, \ldots, A_{\chi(G_2)-1}$, we can connect this $G_1$ and $A_{\chi(G_2)}$ by a path of color~2.
Thus there is a connected monochromatic graph of color~2 containing $G_1\cup G_2$ as a subgraph, a contradiction.
\end{proof}

By Claim~\ref{cl:union-1-2}, all of $A_1, \ldots, A_{\chi(G_2)}$ are of the same color.
Then, by an argument analogous to that in the last paragraph of the proof of (i), we can also deduce a contradiction.
The proof of (ii) is complete.
\end{proof}

We shall also use the following embedding result.

\begin{observation}\label{obs:embedding}
Let $G_1, G_2$ be two graphs of chromatic number 3.
Let $s_1, s_2, s_3$ $($resp., $t_1, t_2, t_3$$)$ be the sizes of color classes in a proper 3-vertex-coloring of $G_1$ $($resp., $G_2$$)$ with $s_1\geq s_2\geq s_3$ $($resp., $t_1\geq t_2\geq t_3$$)$.
If $x\geq s_1+t_1$, $y\geq s_2+t_2$ and $z\geq s_3+t_3$, then $K_{x,y,z}$ contains $G_1\cup G_2$ as a subgraph.
\end{observation}

\begin{proof}
Let $F=K_{x,y,z}$ with partite sets $X, Y, Z$, where $|X|=x$, $|Y|=y$ and $|Z|=z$.
We choose disjoint subsets $S_1, T_1\subseteq X$, $S_2, T_2\subseteq Y$ and $S_3, T_3\subseteq Z$ with $|S_i|=s_i$ and $|T_i|=t_i$ for each $i\in [3]$.
Then $F[S_1\cup S_2\cup S_3]$ contains $G_1$ as a subgraph, and $F[T_1\cup T_2\cup T_3]$ contains $G_2$ as a subgraph.
Hence, $F$ contains a copy of $G_1\cup G_2$.
\end{proof}

Now we have all the ingredients needed to prove Theorem~\ref{th:union-1}.
Before presenting our proof, we provide a brief sketch.
We first deal with several easier cases by comparing suitable lower bounds on $R_3(G_1\cup G_2)$ with upper bounds on $R_2(\mathscr{C}(G_1\cup G_2))$, and then applying Lemma~\ref{le:CH}.
The main case is $\chi(G_1)=\chi(G_2)=3$.
Using Lemma~\ref{le:Caseb+}, we obtain the structural partition $V_2, V_3, \ldots, V_k$.
We first show that $V_4\cup \cdots \cup V_k$ must be small, since otherwise the edges between the parts already contain a monochromatic copy of $G_1\cup G_2$.
This forces $V_2$ to be sufficiently large, and we then apply Ramsey-type results involving the decomposition families of $G_1$ and $G_2$ to find two vertex-disjoint monochromatic copies of $G_1$ and $G_2$.

\vspace{0.3cm}
\noindent{\bf Proof of Theorem~\ref{th:union-1}.}~
Since $5>\frac{5+\sqrt{13}}{2}=\frac{1+4+\sqrt{(1+4)^2-12}}{2}$, the case $\chi(G_2)\geq 5$ follows from Corollary~\ref{cor:union-1}.
If $\chi(G_2)\leq 2$, then $G_1\cup G_2$ is a bipartite graph, so the result follows from Theorem~\ref{th:LBMWW-1} or \ref{th:chro}.
If $\chi(G_2)=4$, then by Lemmas~\ref{le:R2}~(i), \ref{le:R3}~(i), \ref{le:union-1}~(i), we have
\begin{align*}
  &~R_3(G_1\cup G_2)-R_2(\mathscr{C}(G_1\cup G_2)) \\
  \geq &~(\chi(G_2)-1)(R_2(G_2)-1)+\sigma(G_2)-(R_2(G_2)+\chi(G_2)|V(G_2)|) \\
  = &~(\chi(G_2)-2)R_2(G_2)-(\chi(G_2)-1)+\sigma(G_2)-\chi(G_2)|V(G_2)| \\
  \geq &~(\chi(G_2)-2)((\chi(G_2)-1)(|V(G_2)|-1)+\sigma(G_2))-(\chi(G_2)-1)+\sigma(G_2)-\chi(G_2)|V(G_2)| \\
  \geq &~2(3(|V(G_2)|-1)+1)-3+1-4|V(G_2)|  = 2|V(G_2)|-6 ~> 0,
\end{align*}
where the last line follows from $|V(G_2)|\geq \chi(G_2)=4$ and $\sigma(G_2)\geq 1$.
Thus $R_3(G_1\cup G_2)> R_2(\mathscr{C}(G_1\cup G_2))$, so the result follows from Lemma~\ref{le:CH}.

It remains to consider the case $\chi(G_2)=3$.
If $\chi(G_1)\leq 2$ and $\sigma(G_2)=1$, then the result follows from Corollary~\ref{cor:sigma1}.
If $\chi(G_1)\leq 2$ and $\sigma(G_2)\geq 2$, then by Lemmas~\ref{le:R2}~(i), \ref{le:R3}~(i) and \ref{le:union-1}~(ii), we have
\begin{align*}
  &~R_3(G_1\cup G_2)-R_2(\mathscr{C}(G_1\cup G_2)) \\
  \geq &~(\chi(G_2)-1)(R_2(G_2)-1)+\sigma(G_2)-(R_2(G_2)+(\chi(G_2)-1)|V(G_2)|) \\
  = &~2(R_2(G_2)-1)+\sigma(G_2)-(R_2(G_2)+2|V(G_2)|) = R_2(G_2)-2+\sigma(G_2)-2|V(G_2)| \\
  \geq &~(2(|V(G_2)|-1)+\sigma(G_2))-2+\sigma(G_2)-2|V(G_2)| = 2\sigma(G_2)-4 ~\geq 0.
\end{align*}
Thus $R_3(G_1\cup G_2)\geq R_2(\mathscr{C}(G_1\cup G_2))$, so the result follows from Lemma~\ref{le:CH}.
\vspace{0.1cm}

{\bf The remainder of the proof is devoted to the case $\chi(G_1)=\chi(G_2)=3$.}
Let $s_1, s_2, s_3$ (resp., $t_1, t_2, t_3$) be the sizes of color classes in a proper 3-vertex-coloring of $G_1$ (resp., $G_2$) with $s_1\geq s_2\geq s_3=\sigma(G_1)$ (resp., $t_1\geq t_2\geq t_3=\sigma(G_2)$).
By Inequality~(\ref{eq:lower bound}), it suffices to show that $f(G_1\cup G_2,P_5)\leq R_{3}(G_1\cup G_2)$.
For a contradiction, suppose that $F$ is an edge-colored $K_n$ without rainbow $P_5$ or monochromatic $G_1\cup G_2$, where $n=R_{3}(G_1\cup G_2)$.
By Lemma~\ref{le:Caseb+}, we can partition $V(F)$ into $k-1$ parts $V_2, V_3, \ldots, V_{k}$ for some $k\geq 4$ satisfying
\begin{itemize}
\item[{\rm (i)}] $\{i\}\subseteq C(V_i)\subseteq \{1,i\}$ for every $i\in \{2, 3, \ldots, k\}$,
\item[{\rm (ii)}] $c(V_i, V_j)=1$ for every $2\leq i<j\leq k$,
\end{itemize}
and in addition, we have
$|V_2|\geq |V_3|\geq \cdots \geq |V_{k}|,$
\begin{equation}\label{eq:union-1-1}
|V_2|\geq |V_3|\geq |V(G_2)|,
\end{equation}
and
\begin{equation}\label{eq:union-1-2}
|V_4\cup \cdots \cup V_k|\geq |V_4|\geq |V(G_1)|.
\end{equation}

\begin{claim}\label{cl:union-1+1}
$|V_4\cup \cdots \cup V_k|\leq s_3+t_3-1.$
\end{claim}

\begin{proof}
Suppose $|V_4\cup \cdots \cup V_k|\geq s_3+t_3.$
Let $x=\max\{|V_2|, |V_4\cup \cdots \cup V_k|\}$, $y=|V_3|$ and $z=\min\{|V_2|, |V_4\cup \cdots \cup V_k|\}$.
By Inequality~(\ref{eq:union-1-1}) and since $G_1\subseteq G_2$, we have $y\geq |V(G_2)|\geq \frac{1}{2}|V(G_1)|+\frac{1}{2}|V(G_2)|\geq s_2+t_2$ and $z\geq \min\{|V(G_2)|, s_3+t_3\}\geq s_3+t_3,$ where the last inequality follows from $s_3+t_3\leq \frac{1}{3}|V(G_1)|+\frac{1}{3}|V(G_2)|\leq \frac{2}{3}|V(G_2)|$.
Note that $x+y+z=n=R_{3}(G_1\cup G_2)$, $x\geq y$, $x\geq z$ and $x$ is an integer.
Combining with Lemma~\ref{le:R3}~(iii), we have
\begin{align*}
  x\geq &~\left\lceil\frac{1}{3}R_{3}(G_1\cup G_2)\right\rceil ~\geq \left\lceil\frac{1}{3}\left(3(|V(G_1)|+|V(G_2)|)-2\right)\right\rceil = |V(G_1)|+|V(G_2)| ~\geq s_1+t_1.
\end{align*}
By Observation~\ref{obs:embedding}, there is a monochromatic $G_1\cup G_2$ of color~1 using edges between $V_2$, $V_3$ and $V_4\cup \cdots \cup V_k$, a contradiction.
\end{proof}

By Lemma~\ref{le:R3}~(i), Claim~\ref{cl:union-1+1} and Inequalities~(\ref{eq:union-1-1}) and (\ref{eq:union-1-2}), we have
\begin{align}\label{eq:union-1-3}
|V_2|\geq &~ \left\lceil\frac{1}{2}(n-|V_4\cup \cdots \cup V_k|)\right\rceil ~\geq \left\lceil\frac{1}{2}(R_{3}(G_1\cup G_2)-(s_3+t_3-1))\right\rceil \nonumber\\
\geq &~\left\lceil\frac{1}{2}(2(R_2(G_2)-1)+\sigma(G_1)+\sigma(G_2)-(s_3+t_3-1))\right\rceil = \left\lceil\frac{1}{2}(2R_2(G_2)-1)\right\rceil ~=R_2(G_2)
\end{align}
and
\begin{align}\label{eq:union-1-4}
|V(G_1)|= &~s_1+s_2+s_3\leq \frac{3}{2}(s_1+s_2) ~= \frac{3}{2}(|V(G_1)|-s_3) ~\leq \frac{3}{2}(|V_4\cup \cdots \cup V_k|-s_3) \nonumber\\
\leq &~\frac{3}{2}(s_3+t_3-1-s_3) ~= \frac{3}{2}(t_3-1)~\leq \frac{3}{2}\left(\frac{1}{3}|V(G_2)|-1\right) ~= \frac{1}{2}(|V(G_2)|-3). \nonumber
\end{align}
In particular, we have $|V(G_1)|\leq |V(G_2)|-1$.
Combining with Lemma~\ref{le:R2}~(i) and Inequality~(\ref{eq:union-1-3}), we have
\begin{equation}\label{eq:union-1-5}
|V_2|\geq R_2(G_2)\geq 2(|V(G_2)|-1)+1\geq |V(G_1)|+|V(G_2)|.
\end{equation}
Recall that $\mathcal{M}(H)$ is the decomposition family of a graph $H$.

\begin{claim}\label{cl:union-1+2}
$R_2(G_2)\geq R(G_1\cup G_2, \mathcal{M}(G_2))\geq R(G_1\cup G_2, \mathcal{M}(G_1)).$
\end{claim}

\begin{proof}
On the one hand, we have $R_2(G_2)\geq R(G_1, G_2)$ since $G_1\subseteq G_2$.
On the other hand, by Lemma~\ref{le:R-decom}~(i) and since $|V(G_1)|\leq |V(G_2)|-1$, we have
$R_2(G_2)\geq R(G_2, \mathcal{M}(G_2))+|V(G_2)|-1 \geq R(G_2, \mathcal{M}(G_2))+|V(G_1)|.$
Hence, combining with Lemma~\ref{le:R-decom}~(ii), we have $R_2(G_2)\geq \max\{R(G_1, G_2), R(G_2, \mathcal{M}(G_2))+|V(G_1)|\}\geq R(G_1\cup G_2, \mathcal{M}(G_2)).$

Note that if $M\in \mathcal{M}(G_2)$, then embedding a copy of $M$ into one partite set of a complete bipartite graph of large partite sets creates a $G_2$ (and thus a $G_1$).
Thus every $M\in \mathcal{M}(G_2)$ contains some graph in $\mathcal{M}(G_1)$ as a subgraph.
Therefore, $R(G_1\cup G_2, \mathcal{M}(G_2))\geq R(G_1\cup G_2, \mathcal{M}(G_1))$.
\end{proof}

By Claim~\ref{cl:union-1+2} and Inequality~(\ref{eq:union-1-3}), we have $|V_2|\geq R_2(G_2)\geq R(G_1\cup G_2, \mathcal{M}(G_2))=R(\mathcal{M}(G_2), G_1\cup G_2).$
Since $F$ contains no monochromatic $G_1\cup G_2$ of color~2, $F[V_2]$ contains a monochromatic copy $M$ of some graph in $\mathcal{M}(G_2)$ of color~1.
Let $U_1\subseteq V_2$ with $V(M)\subseteq U_1$ and $|U_1|=|V(G_2)|$.
Since $M$ contains at least one edge, we have $|V(M)|\geq 2$.

\begin{claim}\label{cl:union-1+3}
$|V_3|\leq |V(G_2)|-3+s_3.$
\end{claim}

\begin{proof}
Suppose that $|V_3|\geq |V(G_2)|-2+s_3.$
We choose two disjoint subsets $W_1, W_2\subseteq V_3$ with $|W_1|=|V(G_2)|-2$ and $|W_2|=s_3$.
Since $M\subseteq F[U_1]$, $|V(M)|\geq 2$ and $c(U_1, W_1)=1$, $F[U_1\cup W_1]$ contains a monochromatic $G_2$ of color~1.
By Inequalities~(\ref{eq:union-1-2}) and (\ref{eq:union-1-5}), we have $|V_4|\geq |V(G_1)|$ and $|V_2\setminus U_1|\geq |V(G_1)|$.
Thus $F[(V_2\setminus U_1)\cup W_2\cup V_4]$ contains a monochromatic $G_1$ of color~1.
Then $F$ contains a monochromatic $G_1\cup G_2$ of color~1, a contradiction.
\end{proof}

Note that $s_3=\sigma(G_1)\leq \sigma(G_2)=t_3$, since the restriction of any proper 3-vertex-coloring of $G_2$ to $V(G_1)$ is a proper 3-vertex-coloring of $G_1$.
By Lemma~\ref{le:R3}~(i), Lemma~\ref{le:R2}~(i) and Claims~\ref{cl:union-1+1} and \ref{cl:union-1+3}, we have
\begin{align*}
|V_2|= &~ R_{3}(G_1\cup G_2)-|V_3|-|V_4\cup \cdots \cup V_k| \\
\geq &~2(R_2(G_2)-1)+\sigma(G_1)+\sigma(G_2)-(|V(G_2)|-3+s_3)-(s_3+t_3-1) \\
= &~2R_2(G_2)+2-|V(G_2)|-s_3 \\
\geq &~R_2(G_2) + (2(|V(G_2)|-1)+\sigma(G_2))+2-|V(G_2)|-s_3 \\
= &~R_2(G_2)+|V(G_2)|+\sigma(G_2)-s_3 ~\geq R_2(G_2)+|V(G_2)|.
\end{align*}
Thus $|V_2\setminus U_1|\geq R_2(G_2)$.
By Claim~\ref{cl:union-1+2}, we further have $|V_2\setminus U_1|\geq R(G_1\cup G_2, \mathcal{M}(G_1))=R(\mathcal{M}(G_1), G_1\cup G_2).$
Since $F$ contains no monochromatic $G_1\cup G_2$ of color~2, $F[V_2\setminus U_1]$ contains a monochromatic copy $M'$ of some graph in $\mathcal{M}(G_1)$ of color~1.
By Inequality~(\ref{eq:union-1-5}), we can choose $U_2\subseteq V_2\setminus U_1$ with $V(M')\subseteq U_2$ and $|U_2|=|V(G_1)|$.
Recall that $|V_3|\geq |V(G_2)|$ and $|V_4|\geq |V(G_1)|$.
Thus $F[U_1\cup V_3]$ (resp., $F[U_2\cup V_4]$) contains a monochromatic $G_2$ (resp., $G_1$) of color~1.
Then $F$ contains a monochromatic $G_1\cup G_2$ of color~1, a contradiction.
This completes the proof of Theorem~\ref{th:union-1}.
\hfill$\square$

\section{Proofs of results for graphs of chromatic number 3}
\label{sec:proof-3}

In this section, we present our proofs of Theorems~\ref{th:critical}, \ref{th:balanced} and \ref{th:one3}.
We will repeatedly use the following embedding result.

\begin{observation}\label{obs:embedding-3-critical}
Let $G$ be a 3-color-critical graph with a critical edge $e$.
Let $(A,B)$ be a bipartition of the bipartite graph obtained from $G$ by deleting $e$, where $|A|=a$, $|B|=b$, and both endpoints of $e$ lie in $A$.
Let $F$ be an edge-colored complete graph.
Then $F$ contains a monochromatic copy of $G$ if either of the following conditions holds.
\begin{itemize}
\item[{\rm (i)}] There exist an edge $uv$ and disjoint sets $U, W\subseteq V(F)\setminus \{u, v\}$ such that $|U|=a-2$, $|W|=b$, and $c(\{u,v\}\cup U, W)=c(uv)$.
\item[{\rm (ii)}] There exist a vertex $u$ and disjoint sets $U, W\subseteq V(F)\setminus \{u\}$ with $|U|=a-1$, $|W|=b$ such that all edges joining two distinct sets among $\{u\}, U, W$ have the same color.
\end{itemize}
\end{observation}

\begin{proof}
(i) We map the endpoints of $e$ to $u$ and $v$, respectively, map the remaining vertices of $A$ injectively to $U$, and map $B$ injectively to $W$.
Since $c(\{u,v\}\cup U, W)=c(uv)$, we obtain a monochromatic copy of $G$.

(ii) We map one endpoint of $e$ to $u$, map the remaining $a-1$ vertices of $A$ injectively to $U$, and map $B$ injectively to $W$.
Since  all edges joining two distinct sets among $\{u\}, U, W$ have the same color, we obtain a monochromatic copy of $G$.
\end{proof}

Before presenting our proof of Theorem~\ref{th:critical}, we first provide a brief sketch.
Let $H=G_1\cup \cdots \cup G_t$, where the first $s$ components are $3$-color-critical and the remaining components are bipartite.
For each $3$-color-critical component, deleting a critical edge gives a bipartite graph, and Observation~\ref{obs:embedding-3-critical} provides two useful ways to embed such a component.
Using Lemma~\ref{le:Caseb+}, we obtain the structural partition $V_2, V_3, \ldots, V_k$ in which every edge between distinct parts has color~$1$.
We divide the proof according to whether $|V_2|\geq |V(H)|$.
In the case when $|V_2|\geq |V(H)|$, we use Ramsey numbers for matchings to find suitable edges of color~1, and a sequence of estimates allows us to embed all components of $H$ by using
Observation~\ref{obs:embedding-3-critical}.
In the case when $|V_2|\leq |V(H)|-1$, we show that $H$ is a subgraph of a suitable complete 3-partite graph $K_{x,y,s}$, and the size estimates force three unions of the parts $V_2, V_3, \ldots, V_k$ that contain a monochromatic copy of $K_{x,y,s}$.

\vspace{0.3cm}
\noindent{\bf Proof of Theorem~\ref{th:critical}.}~
By Theorem~\ref{th:LBMWW-1} or \ref{th:chro}, it suffices to assume that $\chi(H)=3$, that is, $H$ contains at least one 3-color-critical component.
Let $H=G_1\cup \cdots \cup G_t$, where $G_1, \ldots, G_s$ are $s$ connected 3-color-critical graphs, and $G_{s+1}, \ldots, G_t$ are $t-s$ connected bipartite graphs for some $1\leq s\leq t$.
By Proposition~\ref{prop:isolated}, it suffices to assume that $|V(G_i)|\geq 2$ for all $i\in [t]$.
Note that for $i\in [s]$, the graph $G_i$ can be obtained from a bipartite graph by adding one edge within one partite set.
Let $a_i$ and $b_i$ be the sizes of the two partite sets, and assume that the critical edge is contained in the partite set of size $a_i$.
For $i \in [t]\setminus [s]$, we can also assume that $a_i$ and $b_i$ are the sizes of the two partite sets of $G_i$.
Hence, we have $a_i\geq 2$ for $i\in [s]$, $a_i\geq 1$ for $i\in [t]\setminus [s]$ and $b_i\geq 1$ for $i\in [t]$.
Moreover, we have $\operatorname{oct}(G_i)\geq 1$ for each $i\in [s]$.
Then by Lemma~\ref{le:R3}~(iv), we have
\begin{equation}\label{eq:critical-Ram}
R_3(H)\geq 2|V(H)|+s-2.
\end{equation}

By Inequality~(\ref{eq:lower bound}), it suffices to show that $f(H,P_5)\leq R_{3}(H)$.
For a contradiction, suppose that there exists an edge-coloring $F$ of $K_{R_3(H)}$ without rainbow $P_5$ or monochromatic $H$.
By Lemma~\ref{le:Caseb+}, we can partition $V(F)$ into $k-1$ parts $V_2, V_3, \ldots, V_{k}$ for some $k\geq 4$ satisfying
\begin{itemize}
\item[{\rm (i)}] $\{i\}\subseteq C(V_i)\subseteq \{1,i\}$ for every $i\in \{2, 3, \ldots, k\}$,
\item[{\rm (ii)}] $c(V_i, V_j)=1$ for every $2\leq i<j\leq k$, and
\item[{\rm (iii)}] $|V_2|\geq |V_3|\geq \cdots \geq |V_k|$, $|V_3|\geq \max_{i\in [t]}|V(G_i)|$, $|V_4|\geq \min_{i\in [t]}|V(G_i)|$ and $|V_3\cup \cdots \cup V_k|\geq |V(H)|=\sum_{i\in [t]}(a_i+b_i)$.
\end{itemize}
We first prove a claim related to the upper bound on the size of $V_2$.

\begin{claim}\label{cl:critical-V2-1}
$|V_2|\leq |V(H)|+s-2.$
\end{claim}

\begin{proof}
Suppose that $|V_2|\geq |V(H)|+s-1.$
Then by Corollary~\ref{cor:matching}, we have $|V_2|\geq R(sK_2, H)$.
Since $F$ contains no monochromatic $H$ of color~2, there must be a monochromatic matching $M=\{u_1v_1, \ldots, u_sv_s\}$ of color~1 within $V_2$.
Since $|V_2\setminus V(M)|=|V_2|-2s\geq |V(H)|-s-1 > \sum_{i\in [s]}(a_i-2)+\sum_{i\in [t]\setminus [s]}a_i$ and $|V_3\cup \cdots \cup V_k|\geq |V(H)|=\sum_{i\in [t]}(a_i+b_i)$,
we can choose disjoint subsets $U_1, \ldots, U_t \subseteq V_2\setminus V(M)$ and $W_1, \ldots, W_t\subseteq V_3\cup \cdots \cup V_k$ such that $|U_i|=a_i-2$ for $i\in [s]$, $|U_i|=a_i$ for $i\in [t]\setminus [s]$ and $|W_i|=b_i$ for $i\in [t]$.
Then by Observation~\ref{obs:embedding-3-critical}~(i), for each $i\in [s]$ (resp., $i\in [t]\setminus [s]$), $F[U_i\cup W_i\cup \{u_i, v_i\}]$ (resp., $F[U_i\cup W_i]$) contains a monochromatic $G_i$ of color~1.
Moreover, these monochromatic subgraphs are pairwise vertex-disjoint.
Hence, $F$ contains a monochromatic $H$ of color~1, a contradiction.
\end{proof}

We divide the rest of the proof into two cases based on the size of $V_2$.
\vspace{0.2cm}

\noindent {\bf Case 1.} $|V_2|\geq |V(H)|.$
\vspace{0.2cm}

By Claim~\ref{cl:critical-V2-1}, we may assume that $|V_2|=|V(H)|+m-1$ for some $1\leq m\leq s-1$.
Then by Corollary~\ref{cor:matching}, we have $|V_2|\geq R(mK_2, H)$.
Since $F$ contains no monochromatic $H$ of color~2, there must be a monochromatic matching $M=\{u_1v_1, \ldots, u_mv_m\}$ of color~1 within $V_2$.
Note that
\begin{align}
|V_2\setminus V(M)| =&~|V_2|-2m \geq|V(H)|-m-1 \nonumber\\
 \geq &~\sum\nolimits_{i\in [m]}(a_i-2)+\sum\nolimits_{i\in [t]\setminus [m]}a_i+\sum\nolimits_{i\in [t]}b_i  \label{eq:critical-V2}\\
 > &~\sum\nolimits_{i\in [m]}(a_i-2)+\sum\nolimits_{i\in [t]\setminus [s]}a_i. \label{eq:critical-V2+}
\end{align}
By (\ref{eq:critical-V2+}), we can choose disjoint subsets $U_1, \ldots, U_m, U_{s+1}, \ldots, U_t \subseteq V_2\setminus V(M)$ with $|U_i|=a_i-2$ for $i\in [m]$ and $|U_i|=a_i$ for $i\in [t]\setminus [s]$.
Let $V'_2\colonequals V_2\setminus (V(M)\cup U_1\cup \cdots \cup U_m\cup U_{s+1}\cup \cdots \cup U_t)$.

We next state and prove three claims related to the size of $V_3$.

\begin{claim}\label{cl:critical-V3-1}
Either $|V_3|\leq \sum_{i\in [t]}b_i-1$ or $|V_3|\geq |V(H)|.$
\end{claim}

\begin{proof}
Suppose that $\sum_{i\in [t]}b_i \leq |V_3| \leq |V(H)|-1.$
By Inequality~(\ref{eq:critical-Ram}), we have
$$|V_4\cup \cdots \cup V_k|=R_3(H)-|V_2|-|V_3|\geq 2|V(H)|+s-2 - (|V(H)|+m-1) - (|V(H)|-1) = s-m.$$
Let $u_{m+1}, \ldots, u_s$ be $s-m$ distinct vertices in $V_4\cup \cdots \cup V_k$.
By Inequality~(\ref{eq:critical-V2}), we can choose $s-m$ disjoint subsets $U_{m+1}, \ldots, U_s$ of $V'_2$ with $|U_i|=a_i-1$ for each $i\in [s]\setminus [m]$.
Moreover, since $|V_3|\geq \sum_{i\in [t]}b_i$, we can choose $t$ disjoint subsets $W_1, \ldots, W_t$ of $V_3$ with $|W_i|=b_i$ for each $i\in [t]$.
Let $F_i\colonequals F[U_i\cup W_i\cup \{u_i, v_i\}]$ for $i\in [m]$, $F_i\colonequals F[U_i\cup W_i\cup \{u_i\}]$ for $i\in [s]\setminus [m]$, and $F_i\colonequals F[U_i\cup W_i]$ for $i\in [t]\setminus [s]$.
Then by Observation~\ref{obs:embedding-3-critical}, for each $i\in [t]$, $F_i$ contains a monochromatic $G_i$ of color~1.
Moreover, $F_1, \ldots, F_t$ are pairwise vertex-disjoint.
Hence, $F$ contains a monochromatic $H$ of color~1, a contradiction.
\end{proof}

\begin{claim}\label{cl:critical-V3-2}
$|V_3|\leq \sum_{i\in [t]}b_i-1.$
\end{claim}

\begin{proof}
Suppose for a contradiction that $|V_3|\geq \sum_{i\in [t]}b_i,$ so $|V_3|\geq |V(H)|$ by Claim~\ref{cl:critical-V3-1}.
Let $m'\colonequals |V_3|+1-|V(H)|$ if $|V_3|\leq |V(H)|+(s-m)-2$, and $m'\colonequals s-m$ if $|V_3|\geq |V(H)|+(s-m)-1$.
In either case, we have $1\leq m'\leq s-m\leq t$, so Corollary~\ref{cor:matching} applies to $F[V_3]$.
Hence, $|V_3|\geq R(m'K_2, H)$.
Since $F$ contains no monochromatic $H$ of color~3, there must be a monochromatic matching $M'=\{u_{m+1}v_{m+1}, \ldots, u_{m+m'}v_{m+m'}\}$ of color~1 within $V_3$.
By Inequality~(\ref{eq:critical-Ram}), if $|V_3|\leq |V(H)|+(s-m)-2$, then
\begin{align*}
|V_4\cup \cdots \cup V_k|= &~R_3(H)-|V_2|-|V_3| \\
\geq &~2|V(H)|+s-2 - (|V(H)|+m-1) - (|V(H)|+m'-1) = s-m-m';
\end{align*}
if $|V_3|\geq |V(H)|+(s-m)-1$, then
$$|V_4\cup \cdots \cup V_k|>0= s-m-m'.$$
Thus there exist $s-m-m'$ distinct vertices $u_{m+m'+1}, \ldots, u_s$ in $V_4\cup \cdots \cup V_k$.
By Inequality~(\ref{eq:critical-V2}), we can choose $s-m$ disjoint subsets $W_{m+1}, \ldots, W_{m+m'}, U_{m+m'+1}, \ldots, U_s$ of $V'_2$ with $|W_i|=b_i$ for $i\in \{m+1, \ldots, m+m'\}$ and $|U_i|=a_i-1$ for $i\in [s]\setminus [m+m']$.
Moreover, since $|V_3|\geq |V(H)|$, we can choose $t$ disjoint subsets $W_1, \ldots, W_m, U_{m+1}, \ldots, U_{m+m'}, W_{m+m'+1}, \ldots, W_t\subseteq V_3\setminus V(M')$ with $|U_i|=a_i-2$ for $i\in \{m+1, \ldots, m+m'\}$ and $|W_i|=b_i$ for $i\in [t]\setminus \{m+1, \ldots, m+m'\}$.
Let $F_i\colonequals F[U_i\cup W_i\cup \{u_i, v_i\}]$ for $i\in [m+m']$, $F_i\colonequals F[U_i\cup W_i\cup \{u_i\}]$ for $i\in [s]\setminus [m+m']$, and $F_i\colonequals F[U_i\cup W_i]$ for $i\in [t]\setminus [s]$.
Then by Observation~\ref{obs:embedding-3-critical}, for each $i\in [t]$, $F_i$ contains a monochromatic $G_i$ of color~1.
Moreover, $F_1, \ldots, F_t$ are pairwise vertex-disjoint.
Hence, $F$ contains a monochromatic $H$ of color~1, a contradiction.
\end{proof}

\begin{claim}\label{cl:critical-V3-3}
$|V_3|\leq s-m-1$.
\end{claim}

\begin{proof}
Suppose that $|V_3|\geq s-m$.
By Inequality~(\ref{eq:critical-Ram}) and Claim~\ref{cl:critical-V3-2}, we have
\begin{align*}
|V_4\cup \cdots \cup V_k|= &~R_3(H)-|V_2|-|V_3| \geq 2|V(H)|+s-2 - (|V(H)|+m-1) - \left(\sum\nolimits_{i\in [t]}b_i-1\right) > s-m.
\end{align*}
Thus we can choose $s-m$ distinct vertices $u_{m+1}, \ldots, u_{s}\in V_3$ and $s-m$ distinct vertices $v_{m+1}, \ldots, v_{s}\in V_4\cup \cdots \cup V_k$ such that $M'\colonequals \{u_{m+1}v_{m+1}, \ldots, u_sv_s\}$ is a monochromatic matching of color~1.
Recall that $|V_3\cup \cdots \cup V_k|\geq |V(H)|$,
so we can choose disjoint subsets $W_1, \ldots, W_m, U_{m+1}, \ldots, U_s, W_{s+1}, \ldots, W_t \subseteq (V_3\cup \cdots \cup V_k)\setminus V(M')$ such that $|U_i|=a_i-2$ for $i\in \{m+1, \ldots, s\}$ and $|W_i|=b_i$ for each $i\in [t]\setminus \{m+1, \ldots, s\}$.
By Inequality~(\ref{eq:critical-V2}), we can choose $s-m$ disjoint subsets $W_{m+1}, \ldots, W_s$ of $V'_2$ with $|W_i|=b_i$ for each $i\in \{m+1, \ldots, s\}$.
Then by Observation~\ref{obs:embedding-3-critical}, for each $i\in [s]$ (resp., $i\in [t]\setminus [s]$), $F[U_i\cup W_i\cup \{u_i, v_i\}]$ (resp., $F[U_i\cup W_i]$) contains a monochromatic $G_i$ of color~1.
Moreover, these monochromatic subgraphs are pairwise vertex-disjoint.
Hence, $F$ contains a monochromatic $H$ of color~1, a contradiction.
\end{proof}

By Claim~\ref{cl:critical-V3-2}, we have $|V_3|\leq \sum_{i\in [t]}b_i-1.$
Recall that $|V_3\cup \cdots \cup V_k|\geq |V(H)|=\sum_{i\in [t]}(a_i+b_i).$
Thus there exists an index $j>3$ such that $|V_3\cup \cdots \cup V_{j-1}|\leq \sum_{i\in [t]}b_i-1$ and $|V_3\cup \cdots \cup V_{j}|\geq \sum_{i\in [t]}b_i$.
We now choose a subset $W$ with $|W|=\sum\nolimits_{i\in [t]}b_i$ and $V_3\cup \cdots \cup V_{j-1}\subseteq W \subseteq V_3\cup \cdots \cup V_{j}.$
Recall that $|V_2|\geq |V_3|\geq \cdots \geq |V_k|$.
Thus by Claim~\ref{cl:critical-V3-3}, we have $|V_3\cup \cdots \cup V_{j}|\leq |V_3\cup \cdots \cup V_{j-1}|+|V_3|\leq \sum_{i\in [t]}b_i-1+s-m-1.$
Then combining with Inequality~(\ref{eq:critical-Ram}), we have
\begin{align*}
|V_{j+1}\cup \cdots \cup V_k|= &~R_3(H)-|V_2|-|V_3\cup \cdots \cup V_{j}| \\
\geq &~2|V(H)|+s-2 - (|V(H)|+m-1) - \left(\sum\nolimits_{i\in [t]}b_i-1+s-m-1\right) \\
= &~|V(H)|-\sum\nolimits_{i\in [t]}b_i+1 ~> \sum\nolimits_{i\in [t]}a_i-s.
\end{align*}
Now we can choose distinct vertices $x_1, \ldots, x_s\in V_2$ and disjoint subsets $W_1, \ldots, W_t \subseteq W$ and $X_1, \ldots, X_t \subseteq V_{j+1}\cup \cdots \cup V_k$ such that $|W_i|=b_i$ for $i\in [t]$, $|X_i|=a_i-1$ for $i\in [s]$, and $|X_i|=a_i$ for $i\in [t]\setminus [s]$.
Then by Observation~\ref{obs:embedding-3-critical}~(ii), for each $i\in [s]$ (resp., $i\in [t]\setminus [s]$), $F[X_i\cup W_i\cup \{x_i\}]$ (resp., $F[X_i\cup W_i]$) contains a monochromatic $G_i$ of color~1.
Moreover, these monochromatic subgraphs are pairwise vertex-disjoint.
Hence, $F$ contains a monochromatic $H$ of color~1, a contradiction.
This completes the proof of Case~1.
\vspace{0.2cm}

\noindent {\bf Case 2.} $|V_2|\leq |V(H)|-1.$
\vspace{0.2cm}

Let $x\colonequals \max\left\{\sum\nolimits_{i\in [t]}a_i -s, \sum\nolimits_{i\in [t]}b_i\right\}$
and $y\colonequals \min\left\{\sum\nolimits_{i\in [t]}a_i -s, \sum\nolimits_{i\in [t]}b_i\right\}.$
Then $x\geq y\geq t\geq s$ and
\begin{equation}\label{eq:critical-x+y}
2y\leq x+y =\sum\nolimits_{i\in [t]}a_i -s+\sum\nolimits_{i\in [t]}b_i=|V(H)|-s.
\end{equation}
Moreover, by Observation~\ref{obs:embedding-3-critical}~(ii), we have $H\subseteq K_{x,y,s}$.

We next state and prove three claims related to the sizes of $V_2$ and $V_3$.

\begin{claim}\label{cl:critical-V2-2}
$|V_2|\leq x-1$.
\end{claim}

\begin{proof}
Suppose that $|V_2|\geq x$.
Note that $|V_3\cup \cdots \cup V_k|\geq |V(H)|\geq y$.
Let $j$ be the minimum index such that $|V_3\cup \cdots \cup V_j|\geq y$.
Note that $|V_3\cup \cdots \cup V_j|\leq |V(H)|-1$.
To see this, if $j=3$, then $|V_3\cup \cdots \cup V_j|=|V_3|\leq |V_2|\leq |V(H)|-1$; if $j\geq 4$, then $|V_j|\leq |V_3|\leq |V_3\cup \cdots \cup V_{j-1}|\leq y-1$, so $|V_3\cup \cdots \cup V_j|\leq 2(y-1)\leq |V(H)|-1$ by Inequality~(\ref{eq:critical-x+y}).
Combining with Inequality~(\ref{eq:critical-Ram}), we have
\begin{align*}
|V_{j+1}\cup \cdots \cup V_{k}|= &~R_3(H)-|V_2|-|V_3\cup \cdots \cup V_{j}| \geq 2|V(H)|+s-2 - (|V(H)|-1) - (|V(H)|-1) ~=s.
\end{align*}
Then there is a monochromatic $K_{x, y, s}$ (and thus a monochromatic $H$) of color~1 between $V_2$, $V_3\cup \cdots \cup V_{j}$ and $V_{j+1}\cup \cdots \cup V_{k}$, a contradiction.
\end{proof}

\begin{claim}\label{cl:critical-V2V3}
$|V_2|+|V_3|\leq x+s-1$.
\end{claim}

\begin{proof}
Suppose that $|V_2|+|V_3|\geq x+s\geq \max\left\{\sum\nolimits_{i\in [t]}a_i, \sum\nolimits_{i\in [t]}b_i+s\right\}$.
Then by Claim~\ref{cl:critical-V2-2}, we have $|V_3|\geq x+s-(x-1)\geq s+1$.
Let $u_1, \ldots, u_s\in V_3$ be $s$ distinct vertices.

We claim that $|V_4\cup \cdots \cup V_k|\leq \sum\nolimits_{i\in [t]}b_i-1$, $x=\sum\nolimits_{i\in [t]}b_i$ and $y=\sum\nolimits_{i\in [t]}a_i -s$.
To see this, we first suppose that $|V_4\cup \cdots \cup V_k|\geq \sum\nolimits_{i\in [t]}b_i$.
Then we can choose disjoint subsets $W_1, \ldots, W_t\subseteq V_4\cup \cdots \cup V_k$ with $|W_i|=b_i$ for each $i\in [t]$.
Let $v_1, \ldots, v_s \in V_2$ be $s$ distinct vertices, and let $U_1, \ldots, U_t\subseteq (V_2\cup V_3)\setminus \{u_1, \ldots, u_s, v_1, \ldots, v_s\}$ be $t$ disjoint subsets with $|U_i|=a_i-2$ for $i\in [s]$ and $|U_i|=a_i$ for $i\in [t]\setminus [s]$.
Then by Observation~\ref{obs:embedding-3-critical}~(i), for each $i\in [s]$ (resp., $i\in [t]\setminus [s]$), $F[U_i\cup W_i\cup \{u_i, v_i\}]$ (resp., $F[U_i\cup W_i]$) contains a monochromatic $G_i$ of color~1.
Moreover, these monochromatic subgraphs are pairwise vertex-disjoint.
This implies that $F$ contains a monochromatic $H$ of color~1, a contradiction.
Hence, $|V_4\cup \cdots \cup V_k|\leq \sum\nolimits_{i\in [t]}b_i-1$.
On the other hand, by Inequalities~(\ref{eq:critical-Ram}), (\ref{eq:critical-x+y}) and Claim~\ref{cl:critical-V2-2}, we have
$$|V_4\cup \cdots \cup V_k|\geq R_3(H)-2|V_2|\geq 2|V(H)|+s-2 -2(x-1)> y.$$
Thus $y<|V_4\cup \cdots \cup V_k|\leq \sum\nolimits_{i\in [t]}b_i-1.$
Then since $\{x,y\}=\left\{\sum\nolimits_{i\in [t]}a_i -s, \sum\nolimits_{i\in [t]}b_i\right\}$,
we have $\sum\nolimits_{i\in [t]}b_i =x$ and $\sum\nolimits_{i\in [t]}a_i -s =y<|V_4\cup \cdots \cup V_k|$.

Now we have
\begin{align}\label{eq:critical-1}
|V_2|+|V_3|= &~ R_3(H)-|V_4\cup \cdots \cup V_k| ~\geq 2|V(H)|+s-2 -\left(\sum\nolimits_{i\in [t]}b_i-1\right) \nonumber\\
= &~2\sum\nolimits_{i\in [t]}a_i + \sum\nolimits_{i\in [t]}b_i +s-1.
\end{align}
We next show that $|V_2|\leq \sum\nolimits_{i\in [s]}b_i-1$.
Suppose for a contradiction that $|V_2|\geq \sum\nolimits_{i\in [s]}b_i$.
Recall that $|V_2|+|V_3|\geq x+s \geq \sum\nolimits_{i\in [t]}b_i+s$.
Thus we can choose disjoint subsets $W_1, \ldots, W_s\subseteq V_2$ and $W_{s+1}, \ldots, W_{t}\subseteq (V_2\cup V_3)\setminus (W_1\cup \cdots \cup W_s \cup \{u_1, \ldots, u_s\})$ with $|W_i|=b_i$ for $i\in [t]$.
Since $|V_4\cup \cdots \cup V_k|>\sum\nolimits_{i\in [t]}a_i -s$, we can choose $s$ distinct vertices $v_1, \ldots, v_s \in V_4\cup \cdots \cup V_k$ and $t$ disjoint subsets $U_1, \ldots, U_t\subseteq (V_4\cup \cdots \cup V_k)\setminus \{v_1, \ldots, v_s\}$ with $|U_i|=a_i-2$ for $i\in [s]$ and $|U_i|=a_i$ for $i\in [t]\setminus [s]$.
Then by Observation~\ref{obs:embedding-3-critical}~(i), for each $i\in [s]$ (resp., $i\in [t]\setminus [s]$), $F[U_i\cup W_i\cup \{u_i, v_i\}]$ (resp., $F[U_i\cup W_i]$) contains a monochromatic $G_i$ of color~1.
Moreover, these monochromatic subgraphs are pairwise vertex-disjoint.
This implies that $F$ contains a monochromatic $H$ of color~1, a contradiction.

Now by Lemma~\ref{le:R3}~(iii), we have
\begin{align*}
|V_4\cup \cdots \cup V_k|= &~R_3(H)-|V_2|-|V_3| ~\geq R_3(G_1\cup \cdots \cup G_s)-2|V_2| \\
\geq &~3\sum\nolimits_{i\in [s]}(a_i + b_i)-2 - 2\left(\sum\nolimits_{i\in [s]}b_i-1\right) = 3\sum\nolimits_{i\in [s]}a_i + \sum\nolimits_{i\in [s]}b_i.
\end{align*}
Moreover, by Inequality~(\ref{eq:critical-1}), we have $|V_2|\geq \frac{1}{2}(|V_2|+|V_3|)> \sum\nolimits_{i\in [t]}a_i.$
Recall that $|V_3\cup \cdots \cup V_k|\geq |V(H)|=\sum_{i\in [t]}(a_i+b_i)$.
Thus we can choose $s$ distinct vertices $v_1, \ldots, v_s \in V_2$, $t$ disjoint subsets $U_1, \ldots, U_t\subseteq V_2\setminus \{v_1, \ldots, v_s\}$, $s$ disjoint subsets $W_1, \ldots, W_s\subseteq V_4\cup \cdots \cup V_k$ and $t-s$ subsets $W_{s+1}, \ldots, W_t\subseteq (V_3\cup \cdots \cup V_k)\setminus (W_1\cup \cdots \cup W_s \cup \{u_1, \ldots, u_s\})$ with $|U_i|=a_i-2$ for $i\in [s]$, $|U_i|=a_i$ for $i\in [t]\setminus [s]$ and $|W_i|=b_i$ for $i\in [t]$.
Then by Observation~\ref{obs:embedding-3-critical}~(i), for each $i\in [s]$ (resp., $i\in [t]\setminus [s]$), $F[U_i\cup W_i\cup \{u_i, v_i\}]$ (resp., $F[U_i\cup W_i]$) contains a monochromatic $G_i$ of color~1.
Moreover, these monochromatic subgraphs are pairwise vertex-disjoint.
This implies that $F$ contains a monochromatic $H$ of color~1.
This contradiction completes the proof of Claim~\ref{cl:critical-V2V3}.
\end{proof}

\begin{claim}\label{cl:critical-V2-3}
$|V_2|\leq s-1$.
\end{claim}

\begin{proof}
Suppose that $|V_2|\geq s.$
We first show that $|V_2|\leq \sum\nolimits_{i\in [t]}a_i-s-1$.
To see this, if $|V_2|\geq \sum\nolimits_{i\in [t]}a_i-s$, then let $j$ be the minimum index with $|V_3\cup \cdots \cup V_j|\geq \sum\nolimits_{i\in [t]}b_i$.
Note that if $j\geq 4$, then $|V_3\cup \cdots \cup V_{j-1}|\leq \sum\nolimits_{i\in [t]}b_i-1.$
By Claim~\ref{cl:critical-V2V3}, we have
\begin{align*}
|V_2\cup \cdots \cup V_j|\leq &~|V_2|+\max\left\{|V_3|, \sum\nolimits_{i\in [t]}b_i-1+|V_j|\right\} \\
\leq &~|V_2|+\sum\nolimits_{i\in [t]}b_i-1+|V_3| ~\leq x+s+\sum\nolimits_{i\in [t]}b_i-2.
\end{align*}
Then
\begin{align*}
|V_{j+1}\cup \cdots \cup V_k|= &~R_3(H)-|V_2\cup \cdots \cup V_j| \geq 2|V(H)|+s-2- \left(x+s+\sum\nolimits_{i\in [t]}b_i-2\right) ~> s.
\end{align*}
Now there is a monochromatic $K_{x, y, s}$ (and thus a monochromatic $H$) of color~1 between $V_2$, $V_3\cup \cdots \cup V_{j}$ and $V_{j+1}\cup \cdots \cup V_{k}$, a contradiction.
Hence, $|V_2|\leq \sum\nolimits_{i\in [t]}a_i-s-1$.

Note that $|V_2\cup \cdots \cup V_{k}|\geq |V(H)|$.
Let $\ell$ be the minimum index with $|V_2\cup \cdots \cup V_{\ell}|\geq \sum\nolimits_{i\in [t]}a_i$.
Then $\ell\geq 3$, $|V_2\cup \cdots \cup V_{\ell-1}|\leq \sum\nolimits_{i\in [t]}a_i-1$, and $|V_3\cup \cdots \cup V_{\ell}|=|V_2\cup \cdots \cup V_{\ell}|-|V_2|\geq \sum\nolimits_{i\in [t]}a_i-(\sum\nolimits_{i\in [t]}a_i-s-1)= s+1$.
Combining with Claim~\ref{cl:critical-V2-2}, we have
\begin{align*}
|V_{\ell+1}\cup \cdots \cup V_k|= &~R_3(H)-|V_2\cup \cdots \cup V_{\ell-1}|-|V_{\ell}| ~\geq R_3(H)-|V_2\cup \cdots \cup V_{\ell-1}|-|V_2| \\
\geq &~2|V(H)|+s-2- \left(\sum\nolimits_{i\in [t]}a_i-1\right)- (x-1) ~> \sum\nolimits_{i\in [t]}b_i.
\end{align*}
Let $M=\{u_1v_1, \ldots, u_sv_s\}$ be a monochromatic matching of color~1 with $u_1, \ldots, u_s\in V_2$ and $v_1, \ldots, v_s\in V_3\cup \cdots \cup V_{\ell}$.
Let $U_1, \ldots, U_t\subseteq (V_2\cup \cdots \cup V_{\ell})\setminus V(M)$ and $W_1, \ldots, W_t\subseteq V_{\ell+1}\cup \cdots \cup V_k$ be disjoint subsets with $|U_i|=a_i-2$ for $i\in [s]$, $|U_i|=a_i$ for $i\in [t]\setminus [s]$ and $|W_i|=b_i$ for $i\in [t]$.
Then by Observation~\ref{obs:embedding-3-critical}~(i), for each $i\in [s]$ (resp., $i\in [t]\setminus [s]$), $F[U_i\cup W_i\cup \{u_i, v_i\}]$ (resp., $F[U_i\cup W_i]$) contains a monochromatic $G_i$ of color~1.
Moreover, these monochromatic subgraphs are pairwise vertex-disjoint.
This implies that $F$ contains a monochromatic $H$ of color~1.
This contradiction completes the proof of Claim~\ref{cl:critical-V2-3}.
\end{proof}

Let $j_1$ be the minimum index with $|V_2\cup \cdots \cup V_{j_1}|\geq x$.
Combining with Claim~\ref{cl:critical-V2-3}, we have $j_1\geq 3$, $|V_2\cup \cdots \cup V_{j_1-1}|\leq x-1$ and $|V_{j_1}|\leq |V_2|\leq s-1$.
Then $|V_2\cup \cdots \cup V_{j_1}|\leq x+s-1$.
By (\ref{eq:critical-Ram}) and (\ref{eq:critical-x+y}), we have $|V_2\cup \cdots \cup V_{k}|=R_3(H)\geq 2|V(H)|+s-2>x+y+s.$
Thus we may assume that $j_2$ is the minimum index with $|V_{j_1+1}\cup \cdots \cup V_{j_2}|\geq y$.
Note that $j_2\geq j_1+2$ since $|V_{j_1+1}|\leq |V_2|\leq s-1<y$ by Claim~\ref{cl:critical-V2-3}.
Then $|V_{j_1+1}\cup \cdots \cup V_{j_2-1}|\leq y-1$ and $|V_{j_2}|\leq |V_2|\leq s-1$.
By~(\ref{eq:critical-Ram}) and (\ref{eq:critical-x+y}), we further have
\begin{align*}
|V_{j_2+1}\cup \cdots \cup V_k|= &~R_3(H)-|V_2\cup \cdots \cup V_{j_1-1}|-|V_{j_1}|- |V_{j_1+1}\cup \cdots \cup V_{j_2-1}|-|V_{j_2}| \\
\geq &~2|V(H)|+s-2- (x-1)-(s-1)- (y-1)-(s-1) ~\geq s.
\end{align*}
Then there is a monochromatic $K_{x, y, s}$ (and thus a monochromatic $H$) of color~1 between $V_2\cup \cdots \cup V_{j_1}$, $V_{j_1+1}\cup \cdots \cup V_{j_2}$ and $V_{j_2+1}\cup \cdots \cup V_k$, a contradiction.
This completes the proof of Theorem~\ref{th:critical}.
\hfill$\square$
\vspace{0.2cm}

We now present our proof of Theorem~\ref{th:balanced}.

\vspace{0.3cm}
\noindent{\bf Proof of Theorem~\ref{th:balanced}.}~
Let $H\colonequals G_1\cup G_2\cup \cdots \cup G_t$.
By Proposition~\ref{prop:isolated}, we may assume that $H$ contains no isolated vertices.
By Inequality~(\ref{eq:lower bound}), it suffices to show that $f(H,P_5)\leq R_{3}(H)$.
For a contradiction, suppose that $F$ is an edge-colored $K_{R_3(H)}$ using exactly $k$ colors such that it contains no rainbow $P_5$ or monochromatic $H$.
By Lemma~\ref{le:Caseb}, we can partition $V(F)$ into $k-1$ parts $V_2, V_3, \ldots, V_{k}$ satisfying
\begin{itemize}
\item[{\rm (i)}] $\{i\}\subseteq C(V_i)\subseteq \{1,i\}$ for every $i\in \{2, 3, \ldots, k\}$, and
\item[{\rm (ii)}] $c(V_i, V_j)=1$ for every $2\leq i<j\leq k$.
\end{itemize}
Without loss of generality, we may assume that $|V_2|=\max_{2\leq i\leq k}|V_i|$.

Let $F'$ be a $3$-edge-coloring of $K_n$ obtained from $F$ by recoloring all edges of colors in $[k]\setminus [3]$ with color $3$.
Since $n=R_{3}(H)$, there is a monochromatic copy of $H$ in $F'$.
Denote such a monochromatic $H$ by $A$.
Since $F$ contains no monochromatic copy of $H$, the color of $A$ must be $3$ and the components of $A$ must be contained in at least two distinct parts of $V_{3}, V_{4}, \ldots, V_{k}$.
Without loss of generality, let $V_3, \ldots, V_{\ell}$ be the parts containing components of $A$ for some $4\leq \ell\leq k$, and assume that $|V_3|\geq \cdots \geq |V_{\ell}|$.
For each $i\in \{3, \ldots, \ell\}$, let $x_i$ be the number of components of $A$ contained in $V_i$, and let $G_{i_1}, \ldots, G_{i_{x_i}}$ be these components, where $i_1, \ldots, i_{x_i}\in [t]$ are their original indices.
Then $\sum^{\ell}_{i=3}x_i=t$ and $|V_2|\geq |V_i|\geq \sum_{j=1}^{x_i}|V(G_{i_j})|= \sum_{j=1}^{x_i}(a_{i_j}+b_{i_j}+c_{i_j})$ for each $i\in \{3, \ldots, \ell\}$.

By the assumptions of the theorem, we have $a_i\geq b_i$ and $b_i+c_i\geq \sum_{j\in [t]\setminus \{i\}} c_j$ for all $i\in [t]$.
Then
\begin{align*}
|V_2|\geq  &~|V_3| ~\geq \sum\nolimits_{j=1}^{x_3}(a_{3_j}+b_{3_j}+c_{3_j}) ~\geq a_{3_1}+c_{3_1}+\sum\nolimits_{j=1}^{x_3}b_{3_j} ~\geq b_{3_1}+c_{3_1}+\sum\nolimits_{j=1}^{x_3}b_{3_j}\\
\geq &~\sum\nolimits_{j\in [t]\setminus \{3_1\}} c_j+\sum\nolimits_{j=1}^{x_3}b_{3_j} ~\geq \sum\nolimits_{i=4}^{\ell}\sum\nolimits_{j=1}^{x_i}c_{i_j}+\sum\nolimits_{j=1}^{x_3}b_{3_j}.
\end{align*}
Thus we can choose disjoint subsets $V_{2,2}, U_4, \ldots, U_{\ell}$ of $V_2$ with $|V_{2,2}|=\sum\nolimits_{j=1}^{x_3}b_{3_j}$ and $|U_i|=\sum\nolimits_{j=1}^{x_i}c_{i_j}$ for each $i\in \{4, \ldots, \ell\}$.
Moreover, by the assumptions of the theorem, we have $b_i\leq a_i\leq \frac{1}{2}(a_i+b_i+c_i)=\frac{1}{2}|V(G_i)|$ for all $i\in [t]$.
Indeed, if $G_i$ is 3-chromatic, then $a_i-b_i\leq 1$ and $c_i\geq 1$, so $a_i\leq b_i+c_i$;
if $G_i$ is bipartite, then $a_i=b_i$ and $c_i=0$.
Combining with $|V_3|\geq \cdots \geq |V_{\ell}|$, we have that for $i\in \{3, \ldots, \ell-1\}$,
\begin{align*}
\sum\nolimits_{j=1}^{x_i}a_{i_j}+\sum\nolimits_{j=1}^{x_{i+1}}b_{(i+1)_j} \leq &~\frac{1}{2}\sum\nolimits_{j=1}^{x_i}|V(G_{i_j})|+\frac{1}{2}\sum\nolimits_{j=1}^{x_{i+1}}|V(G_{(i+1)_j})| \\
\leq &~\frac{1}{2}|V_i|+\frac{1}{2}|V_{i+1}| ~\leq \frac{1}{2}|V_i|+\frac{1}{2}|V_{i}| ~= |V_i|.
\end{align*}
Thus for each $i\in \{3, \ldots, \ell-1\}$, we can choose disjoint subsets $V_{i,1}, V_{i,2}$ of $V_i$ with $|V_{i,1}|=\sum\nolimits_{j=1}^{x_i}a_{i_j}$ and $|V_{i,2}|=\sum\nolimits_{j=1}^{x_{i+1}}b_{(i+1)_j}$.
Furthermore, since $b_i+c_i\geq \sum_{j\in [t]\setminus \{i\}} c_j$ for all $i\in [t]$, we have
\begin{align*}
|V_{\ell}|\geq  &~\sum\nolimits_{j=1}^{x_{\ell}}(a_{\ell_j}+b_{\ell_j}+c_{\ell_j}) ~\geq \left(\sum\nolimits_{j=1}^{x_{\ell}}a_{\ell_j}\right)+(b_{\ell_1}+c_{\ell_1}) \\
\geq &~\sum\nolimits_{j=1}^{x_{\ell}}a_{\ell_j}+\sum\nolimits_{j\in [t]\setminus \{\ell_1\}} c_j ~\geq \sum\nolimits_{j=1}^{x_{\ell}}a_{\ell_j}+\sum\nolimits_{j=1}^{x_3}c_{3_j}.
\end{align*}
Thus we can choose disjoint subsets $V_{\ell,1}, U_3$ of $V_{\ell}$ with $|V_{\ell,1}|=\sum\nolimits_{j=1}^{x_{\ell}}a_{\ell_j}$ and $|U_3|=\sum\nolimits_{j=1}^{x_3}c_{3_j}$.
Now for each $i\in \{3, \ldots, \ell\}$, $F[V_{i,1}\cup V_{i-1,2}\cup U_i]$ contains a monochromatic $G_{i_1}\cup \cdots \cup G_{i_{x_i}}$ of color~1.
These monochromatic subgraphs are pairwise vertex-disjoint and form a monochromatic copy of $H$.
This contradiction completes the proof of Theorem~\ref{th:balanced}.
\hfill$\square$
\vspace{0.2cm}

Before presenting our proof of Theorem~\ref{th:one3}, we first provide a brief sketch.
Let $H_1$ be the unique component of $H$ with chromatic number $3$, and let $H_2, \ldots, H_t$ be the remaining bipartite components.
Using Lemma~\ref{le:Caseb+}, we obtain the structural partition $V_2, V_3, \ldots, V_k$ in which every edge between distinct parts has color~$1$.
Using the size properties of this partition, we first find a monochromatic copy of $H_1$ of color~$1$.
Hence, it remains to find a monochromatic complete bipartite graph $K_{t_1,t_2}$ containing $H_2\cup \cdots \cup H_t$ on vertices disjoint from this copy of $H_1$.
To this end, we establish the lower bound $R_3(H)\geq 2|V(H)|+\sigma(H)-2$ and use the absence of such a $K_{t_1,t_2}$ to obtain upper bounds on $|V_2|$ and $|V_2|+|V_3|$.
Consequently, $V_4\cup\cdots\cup V_k$ is large.
We divide these remaining parts into two unions whose sizes are as balanced as possible, and show that the absence of the required $K_{t_1,t_2}$ contradicts the choice of this partition.

\vspace{0.3cm}
\noindent{\bf Proof of Theorem~\ref{th:one3}.}~
If $H$ is connected, the result follows from Theorem~\ref{th:LBMWW-1}.
Hence, we may assume that $H$ contains at least two components.
By Proposition~\ref{prop:isolated}, we may further assume that $H$ contains no isolated vertices.
Let $H_1, H_2, \ldots, H_t$ be the components of $H$ with $\chi(H_1)=3$, $\chi(H_2)=\cdots =\chi(H_t)=2$ and $\min_{i\in [t]}|V(H_i)|\geq \sigma(H)=\sigma(H_1)$.
Let $s_1, s_2, s_3$ (resp., $t_1, t_2$) be the sizes of color classes in a proper 3-vertex-coloring of $H_1$ (resp., a proper 2-vertex-coloring of $H_2\cup \cdots \cup H_t$) with $s_1\geq s_2\geq s_3=\sigma(H_1)$ (resp., $t_1\geq t_2$).

By Inequality~(\ref{eq:lower bound}), it suffices to show that $f(H,P_5)\leq R_{3}(H)$.
Suppose, for the sake of contradiction, that there exists an edge-coloring $F$ of $K_{R_3(H)}$ without rainbow $P_5$ or monochromatic $H$.
By Lemma~\ref{le:Caseb+}, we can partition $V(F)$ into $k-1$ parts $V_2, V_3, \ldots, V_{k}$ for some $k\geq 4$ satisfying
\begin{itemize}
\item[{\rm (i)}] $\{i\}\subseteq C(V_i)\subseteq \{1,i\}$ for every $i\in \{2, 3, \ldots, k\}$,
\item[{\rm (ii)}] $c(V_i, V_j)=1$ for every $2\leq i<j\leq k$, and
\item[{\rm (iii)}] $|V_2|\geq |V_3|\geq \cdots \geq |V_k|$, $|V_3|\geq \max_{i\in [t]}|V(H_i)|\geq s_1+s_2+s_3$, $|V_4|\geq \min_{i\in [t]}|V(H_i)|\geq \sigma(H_1) = s_3$ and $|V_3\cup \cdots \cup V_k|\geq |V(H)|$.
\end{itemize}

We arbitrarily fix three subsets $X\subseteq V_2$, $Y\subseteq V_3$ and $Z\subseteq V_4\cup \cdots \cup V_k$ with $|X|=s_2$, $|Y|=s_1$ and $|Z|=s_3$.
Then the edges between $X$, $Y$, $Z$ form a monochromatic $K_{s_1, s_2, s_3}$ of color~1.
Let $F'=F-(X\cup Y\cup Z)$.
If $F'$ contains a monochromatic $K_{t_1, t_2}$ of color~1, then $F$ contains a monochromatic $K_{s_1, s_2, s_3}\cup K_{t_1, t_2}$ and thus a monochromatic $H$, a contradiction.
Hence, we may assume that the following statement holds.

\begin{fact}\label{fa:one3}
$F'$ contains no monochromatic $K_{t_1, t_2}$ of color~1.
\end{fact}

We now provide a lower bound on $R_3(H)$.

\begin{claim}\label{cl:one3-R3}
$R_3(H)\geq 2|V(H)|+\sigma(H)-2.$
\end{claim}

\begin{proof}
Let $G$ be a 3-edge-colored complete graph of order $2|V(H)|+\sigma(H)-3$ defined as follows.
We partition $V(G)$ into three subsets $U_1, U_2, U_3$ with $|U_1|=|U_2|=|V(H)|-1$ and $|U_3|=\sigma(H)-1$.
We color all edges within $U_1$ and within $U_3$ with color~1, all edges within $U_2$ with color~2, and all the remaining edges with color~3.
Since $|U_2|<|V(H)|$, $G$ contains no monochromatic $H$ of color~2.
Moreover, since $\min_{i\in [t]}|V(H_i)|\geq \sigma(H)$, $G[U_3]$ contains no monochromatic $H_i$ for any $i\in [t]$.
Combining with $|U_1|<|V(H)|$, there is no monochromatic $H$ of color~1 either.
Finally, the subgraph of $G$ formed by edges of color~3 is a complete 3-partite graph with partite sets $U_1, U_2, U_3$.
If there is a monochromatic $H_1$ of color~3, then each of $U_1, U_2, U_3$ has size at least $\sigma(H_1)=\sigma(H)$, contradicting $|U_3|=\sigma(H)-1$.
Hence, $G$ contains no monochromatic $H$ of color~3.
Therefore, $R_3(H)\geq |V(G)|+1=2|V(H)|+\sigma(H)-2.$
\end{proof}

We next state and prove two claims related to the sizes of $V_2$ and $V_3$.

\begin{claim}\label{cl:one3-1}
$|V_2|\leq s_2+t_2-1.$
\end{claim}

\begin{proof}
Suppose that $|V_2|\geq s_2+t_2.$
Then we can choose $U\subseteq V_2\setminus X$ with $|U|=t_2$.
Since $|V_3\cup \cdots \cup V_k|-|Y\cup Z|\geq |V(H)|-s_1-s_3>t_1$,
the complete bipartite graph between $U$ and $(V_3\cup \cdots \cup V_k)\setminus (Y\cup Z)$ is monochromatic in color~1 and therefore contains a monochromatic $K_{t_1, t_2}$.
This contradicts Fact~\ref{fa:one3}.
\end{proof}

\begin{claim}\label{cl:one3-2}
$|V_2|+|V_3|\leq s_1+s_2+t_2-1.$
\end{claim}

\begin{proof}
Suppose that $|V_2|+|V_3|\geq s_1+s_2+t_2.$
Then we can choose $U\subseteq (V_2\cup V_3)\setminus (X\cup Y)$ with $|U|=t_2$.
By Claims~\ref{cl:one3-R3} and \ref{cl:one3-1}, we have
\begin{align*}
|V_4\cup \cdots \cup V_k|-|Z|= &~R_3(H)-|V_2|-|V_3|-s_3 ~\geq R_3(H)-2|V_2|-s_3\\
\geq &~2|V(H)|+\sigma(H)-2-2(s_2+t_2-1)-s_3 \\
= &~2(s_1+s_2+s_3+t_1+t_2)+s_3-2-2(s_2+t_2-1)-s_3 \\
= &~2(s_1+s_3+t_1) ~> t_1.
\end{align*}
Then the edges between $U$ and $(V_4\cup \cdots \cup V_k)\setminus Z$ form a monochromatic subgraph of color~1 that contains $K_{t_1, t_2}$ as a subgraph.
This contradicts Fact~\ref{fa:one3}.
\end{proof}

By Claims~\ref{cl:one3-R3} and \ref{cl:one3-2}, we have
\begin{align}\label{eq:one3-1}
|V_4\cup \cdots \cup V_k|= &~R_3(H)-|V_2|-|V_3| \geq 2|V(H)|+\sigma(H)-2-(s_1+s_2+t_2-1) \nonumber\\
= &~2(s_1+s_2+s_3+t_1+t_2)+s_3-2-(s_1+s_2+t_2-1) \nonumber\\
= &~s_1+s_2+3s_3+2t_1+t_2-1.
\end{align}
In particular, we have $|V_4\cup \cdots \cup V_k|> |V_2|$ by Claim~\ref{cl:one3-1}.
Thus $k\geq 5$ since $|V_2|\geq |V_3|\geq \cdots \geq |V_k|$.
Choose a subset $I\subseteq \{4, 5, \ldots, k\}$ and correspondingly define $A_1=\bigcup_{i\in I}V_i$ and $A_{2}=(V_4\cup \cdots \cup V_k)\setminus A_1$ such that the following conditions hold:
\begin{itemize}
\item[(1)] $|A_1|\geq |A_{2}|$,
\item[(2)] $|A_1|-|A_2|$ is smallest subject to (1).
\end{itemize}
By Inequality~(\ref{eq:one3-1}) and condition~(1), we have
$$|A_1|\geq \frac{1}{2}|V_4\cup \cdots \cup V_k|\geq \frac{1}{2}(s_1+s_2+3s_3+2t_1+t_2-1).$$
In particular, we have $|A_1|> s_3+t_1.$
Recall that the subset $Z$ was chosen arbitrarily among all subsets of $V_4\cup \cdots \cup V_k$ of sizes $s_3$.
Without loss of generality, we may assume that $Z\subseteq A_1$.
Let $U\subseteq A_1\setminus Z$ with $|U|=t_1$.

If $|A_2|\geq t_2$, then the edges between $U$ and $A_2$ form a monochromatic subgraph of color~1 that contains $K_{t_1, t_2}$ as a subgraph.
This contradicts Fact~\ref{fa:one3}.
Therefore, we may assume that $|A_2|< t_2.$
Combining with Inequality~(\ref{eq:one3-1}), Claim~\ref{cl:one3-1} and $t_2\leq t_1$, we have
$$|A_2|< t_2<\frac{1}{3}(s_1+s_2+3s_3+2t_1+t_2-1)\leq \frac{1}{3}|V_4\cup \cdots \cup V_k|=\frac{1}{3}(|A_1|+|A_2|)$$
and
$$|A_1|=|V_4\cup \cdots \cup V_k|-|A_2|>(s_1+s_2+3s_3+2t_1+t_2-1)-t_2>s_2+t_2>|V_2|=\max\nolimits_{2\leq i\leq k}|V_i|,$$
so $|A_2|< \frac{1}{2}|A_1|$ and $|I|\geq 2$.
Let $j\in I$ with $|V_j|=\min_{i\in I}|V_i|,$ so $|V_j|\leq \frac{1}{2}|A_1|.$
If $|A_1\setminus V_j|\geq |A_2\cup V_j|$, then $I'=I\setminus \{j\}$ is a better choice than $I$ by condition~(2).
Thus $|A_1\setminus V_j|< |A_2\cup V_j|$.
Let $I^{\ast}=(\{4, 5, \ldots, k\}\setminus I)\cup \{j\}$, $A^{\ast}_1=\bigcup_{i\in I^{\ast}}V_i=A_2\cup V_j$ and $A^{\ast}_{2}=(V_4\cup \cdots \cup V_k)\setminus A^{\ast}_1=A_1\setminus V_j$.
Then $|A^{\ast}_1|=|A_2\cup V_j|>|A_1\setminus V_j|=|A^{\ast}_2|$ and
\begin{align*}
(|A_1|-|A_2|)-(|A^{\ast}_1|-|A^{\ast}_2|) = &~(|A_1|-|A_2|)-(|A_2|+|V_j|-(|A_1|-|V_j|)) \\
= &~2(|A_1|-|A_2|-|V_j|) ~>2\left(|A_1|-\frac{1}{2}|A_1|-\frac{1}{2}|A_1|\right) ~=0,
\end{align*}
contradicting the choice of $I$.
This completes the proof of Theorem~\ref{th:one3}.
\hfill$\square$

\section{Concluding remarks}
\label{sec:conclu}

In this paper, we address Conjecture~\ref{conj:P5} for several classes of disconnected graphs with chromatic number at least 3.
Our general results encompass the previously known classes listed in Section~\ref{sec:introduction}.
The main ingredients of our proofs include the structural characterization of Thomason and Wagner, Simonovits' decomposition-family method, Ramsey numbers for families of graphs, and a recoloring argument of Li, Besse, Magnant, Wang and Watts.
In the following, we make some remarks to illustrate the challenges in solving Conjecture~\ref{conj:P5} completely and present several related results and open problems.
In particular, we will prove several results for a bipartite variant of the constrained Ramsey number.
\vspace{0.2cm}

\noindent {\bf (1) Challenges in solving Conjecture~\ref{conj:P5}.}~
First, in our proofs of Theorems~\ref{th:union-1} and \ref{th:union-2}, we utilized Lemma~\ref{le:CH} which states that if $R_3(H)\geq R_2(\mathscr{C}(H))$, then $f(H, P_5)=R_3(H)$.
This lemma provides a powerful sufficient condition to ensure $f(H, P_5)=R_3(H)$.
However, there exist graphs $H$ such that $R_3(H)< R_2(\mathscr{C}(H))$.
For instance, Lorimer and Segedin~\cite{LoSe} proved that $R_3(nK_r)\leq 3(n-1)r+R_3(K_r)$, while Roberts~\cite{Rob} proved that $R_2(\mathscr{C}(nK_r))=(r^2-r+1)n-r+1$ for $n\geq R_2(K_r)$.
Thus $R_3(nK_r)< R_2(\mathscr{C}(nK_r))$ when $r\geq 4$ and $n\geq \max\left\{R_2(K_r), \frac{R_3(K_r)-2r}{r^2-4r+1}\right\}.$
Therefore, if Conjecture~\ref{conj:P5} is true, then the condition $R_3(H)\geq R_2(\mathscr{C}(H))$ cannot be a necessary condition for $f(H, P_5)=R_3(H)$.
In order to address Conjecture~\ref{conj:P5} for graphs $H$ with $R_3(H)< R_2(\mathscr{C}(H))$, new ideas appear necessary.

Second, in our proofs of Theorems~\ref{th:chro}, \ref{th:homology}, \ref{th:critical}, \ref{th:balanced} and \ref{th:one3}, we applied Lemma~\ref{le:Caseb} or \ref{le:Caseb+} which characterize the possible structure of an edge-colored $K_{R_3(H)}$ without a rainbow $P_5$ or a monochromatic $H$.
As in our proofs, upon employing this structural result, we frequently utilize lower bounds of the Ramsey number $R_3(H)$ to derive an upper or lower bound on the sizes of $V_2, V_3, \ldots, V_k$.
The main challenge in this procedure is that we currently lack a sufficiently strong general lower bound on $R_3(H)$ for this purpose.
As research on $R_3(H)$ advances, we believe that more results related to Conjecture~\ref{conj:P5} will become attainable.

Third, for graphs $H$ with $\chi(H)\geq 4$, the problem becomes more difficult.
When we employ the partition $V_2, V_3, \ldots, V_k$ guaranteed by Lemma~\ref{le:Caseb} or \ref{le:Caseb+}, we commonly find a monochromatic copy of $H$ of color~1 consisting of certain edges between the parts.
If $k=4$, it is impossible to find a copy of $H$ (if $\chi(H)\geq 4$) using edges between $V_2$, $V_3$ and $V_4$.
To acquire a monochromatic copy of $H$, we must conduct an in-depth study of the structure within each part.
However, this needs a good understanding of the relationship between $R_3(H)$ and 2-color Ramsey numbers $R(H, H')$ for $H$ and its subgraphs $H'$.
We believe that the relationship between 3-color Ramsey numbers and 2-color Ramsey numbers is a promising direction for future research.
\vspace{0.2cm}

\noindent {\bf (2) A first step toward solving the problem.}~
As the first step towards solving Conjecture~\ref{conj:P5}, we propose the following foundational problem.

\begin{problem}\label{prob:G1G2}
Given two connected graphs $G_1,G_2$ with $\chi(G_1)\leq \chi(G_2)= 3$, determine whether $f(G_1\cup G_2,P_5)=R_3(G_1\cup G_2)$.
\end{problem}

We can prove the following partial result; however, solving the problem in general remains challenging.

\begin{proposition}\label{prop:G1G2+}
Let $G_1,G_2$ be two connected graphs with $\chi(G_1)\leq \chi(G_2)= 3$ and $\min\{|V(G_1)|,$ $|V(G_2)|\}\geq \sigma_3(G_1)+\sigma_3(G_2)$.
Then $f(G_1\cup G_2, P_5)=R_3(G_1\cup G_2).$
\end{proposition}

\begin{proof}
If $\chi(G_1)=2$, then $\sigma_3(G_1)=0$, so $\min\{|V(G_1)|, |V(G_2)|\}\geq \sigma_3(G_1)+\sigma_3(G_2)$ is equivalent to $|V(G_1)|\geq \sigma(G_1\cup G_2)$.
Then the proposition is true by Theorem~\ref{th:one3}.
Assume now $\chi(G_1)=\chi(G_2)= 3$ and $\min\{|V(G_1)|, |V(G_2)|\}\geq \sigma_3(G_1)+\sigma_3(G_2)$.
Let $s_1, s_2, s_3$ (resp., $t_1, t_2, t_3$) be the sizes of color classes in a proper 3-vertex-coloring of $G_1$ (resp., $G_2$) with $s_1\geq s_2\geq s_3=\sigma(G_1)$ (resp., $t_1\geq t_2\geq t_3=\sigma(G_2)$).

By Inequality~(\ref{eq:lower bound}), it suffices to show that $f(G_1\cup G_2, P_5)\leq R_{3}(G_1\cup G_2)$.
Suppose, for the sake of contradiction, that there exists an edge-coloring $F$ of $K_{R_3(G_1\cup G_2)}$ without rainbow $P_5$ or monochromatic $G_1\cup G_2$.
By Lemma~\ref{le:Caseb+}, we can partition $V(F)$ into $k-1$ parts $V_2, V_3, \ldots, V_{k}$ for some $k\geq 4$ satisfying
\begin{itemize}
\item[{\rm (i)}] $\{i\}\subseteq C(V_i)\subseteq \{1,i\}$ for every $i\in \{2, 3, \ldots, k\}$,
\item[{\rm (ii)}] $c(V_i, V_j)=1$ for every $2\leq i<j\leq k$,
\end{itemize}
and in addition, we have
$|V_2|\geq |V_3|\geq \cdots \geq |V_{k}|,$
$$|V_3|\geq \max\{|V(G_1)|, |V(G_2)|\}\geq \frac{1}{2}|V(G_1)|+\frac{1}{2}|V(G_2)|\geq s_2+t_2$$
and
$$\min\{|V_2|, |V_4\cup \cdots \cup V_k|\}\geq |V_4|\geq \min\{|V(G_1)|, |V(G_2)|\}\geq \sigma_3(G_1)+\sigma_3(G_2)= s_3+t_3.$$
Moreover, by Lemma~\ref{le:R3}~(iii) and since $|V_2|\geq |V_3|$ and $|V_2|+|V_3|+|V_4\cup \cdots \cup V_k|=R_3(G_1\cup G_2)$, we have
\begin{align*}
  \max\{|V_2|, |V_4\cup \cdots \cup V_k|\}\geq &~\left\lceil\frac{1}{3}R_{3}(G_1\cup G_2)\right\rceil ~\geq \left\lceil\frac{1}{3}\left(3(|V(G_1)|+|V(G_2)|)-2\right)\right\rceil\\
   = &~|V(G_1)|+|V(G_2)| ~\geq s_1+t_1.
\end{align*}
By Observation~\ref{obs:embedding}, there is a monochromatic $G_1\cup G_2$ of color~1 using edges between $V_2$, $V_3$ and $V_4\cup \cdots \cup V_k$, a contradiction.
This completes the proof of Proposition~\ref{prop:G1G2+}.
\end{proof}
\vspace{0.2cm}

\noindent {\bf (3) Lower bound constructions.}~
For all graphs $H$ for which we have proved so far that $f(H, P_5)=R_{3}(H)$, the corresponding extremal constructions are $3$-edge-colored $K_{R_{3}(H)-1}$ without monochromatic copies of $H$.
This is in fact not related to the condition of rainbow $P_5$, since a $3$-edge-colored graph certainly contains no rainbow $P_5$.
We do not know whether there exists an extremal construction using at least four colors.
We suspect that the problem may behave very differently if we restrict our attention to edge-colorings with at least four colors.
Therefore, we suggest the following problem for further research.

\begin{problem}\label{prob:exact}
Let $H$ be a graph.
\begin{itemize}
\item[{\rm (i)}] Determine the minimum integer $N$ such that for every integer $n\geq N$, in every edge-coloring of $K_{n}$ \textbf{with at least four colors}, there is either a monochromatic copy of $H$ or a rainbow copy of $P_5$.
\item[{\rm (ii)}] For any integer $k\geq 4$, determine the minimum integer $N$ with ${N\choose 2}\geq k$ such that for every integer $n\geq N$, in every edge-coloring of $K_{n}$ \textbf{with exactly $k$ colors}, there is either a monochromatic copy of $H$ or a rainbow copy of $P_5$.
\end{itemize}
\end{problem}

Regarding Problem~\ref{prob:exact}~(ii), we can provide the following lower bound construction when $k \leq \chi(H)$.
Recall that $\mathcal{M}(H)$ is the decomposition family of $H$.
For $2\leq i\leq \chi(H)-1$, let $\mathcal{M}_i(H)$ be the set of minimal graphs that satisfy the following:
for each graph $M\in \mathcal{M}_i(H)$, there exists an integer $t$ such that $H\subseteq \left(M\cup \overline{K}_{t-|V(M)|}\right)\vee K_{(\chi(H)-i)\times t}$.
Note that $\mathcal{M}_2(H)=\mathcal{M}(H)$.
Then for $4\leq k \leq \chi(H)$, a lower bound on the minimum integer $N$ in Problem~\ref{prob:exact}~(ii) is $$(k-2)(|V(H)|-1)+R(\mathcal{M}_{\chi(H)-k+2}(H), H).$$
To see this, we define a $k$-edge-colored complete graph on $(k-2)(|V(H)|-1)+R(\mathcal{M}_{\chi(H)-k+2}(H), H)-1$ vertices as follows.
We partition the vertex set into $k-1$ parts $V_2, V_3, \ldots, V_k$ with $|V_2|=|V_3|= \cdots =|V_{k-1}|=|V(H)|-1$ and $|V_k|=R(\mathcal{M}_{\chi(H)-k+2}(H), H)-1$.
For each $i\in \{2, 3, \ldots, k-1\}$, we color all the edges within $V_i$ using color $i$.
We color the edges within $V_k$ using colors 1 and $k$ such that it contains neither a monochromatic graph in $\mathcal{M}_{\chi(H)-k+2}(H)$ of color~1 nor a monochromatic $H$ of color $k$, and such that color~$k$ is used\footnote{Such a coloring exists. Indeed, take a coloring guaranteed by the definition of $R(\mathcal{M}_{\chi(H)-k+2}(H),H)$. If color~$k$ is not used, recolor an arbitrary edge with color~$k$. This creates neither a monochromatic graph in $\mathcal{M}_{\chi(H)-k+2}(H)$ of color~$1$ nor a monochromatic copy of $H$ of color~$k$, since $|E(H)|>1$ by $\chi(H)\geq4$.}.
We color all the remaining edges with color~1.
Then there is no monochromatic $H$ or rainbow $P_5$ in this edge-colored complete graph using exactly $k$ colors, and thus the lower bound holds.

This lower bound is sharp for $H=K_p$ when $k\in \{p-1, p\}$.
Indeed, it was shown by the authors and Wang~\cite{LiWL} that for $p\geq 5$ and $k\in \{p-1, p\}$, the answer to Problem~\ref{prob:exact}~(ii) is $(k-2)(p-1)+R(K_{p-k+2}, K_p)$.
For $4\leq k\leq p-2$, the first author and Su showed in an unpublished manuscript that the answer to Problem~\ref{prob:exact}~(ii) is $\max_{(a_2, \ldots, a_k)\in X}\big(\sum_{i=2}^{k}(R(K_{a_i+1}, K_p)-1)\big)+1,$ where $X=\big\{(a_2, \ldots, a_k)\in \{1, 2, \ldots, p-k+2\}^{k-1}\colon\, \sum_{i=2}^{k}a_i= p-1\big\}.$
Note that $\max_{(a_2, \ldots, a_k)\in X}\big(\sum_{i=2}^{k}(R(K_{a_i+1}, K_p)-1)\big)+1= (k-2)(p-1)+R(K_{p-k+2}, K_p)$ if the inequality $R(K_a, K_p)+R(K_b, K_p)\leq R(K_{a-1}, K_p)+R(K_{b+1}, K_p)$ holds for all $3\leq a\leq b\leq p-1$.
However, whether the latter inequality holds remains an open problem.
For general graphs, especially disconnected graphs, Problem~\ref{prob:exact} remains open.
\vspace{0.2cm}

\noindent {\bf (4) Longer paths.}~
For a general path $P_t$ with $t\geq 6$, we have $f(H,P_t)\geq R_{t-2}(H)$ since a $(t-2)$-edge-colored graph trivially contains no rainbow $P_t$.
However, we have no convincing evidence to suggest that $f(H,P_t)\leq R_{t-2}(H)$ holds true.
Hence, we pose the following problem for further research.

\begin{problem}\label{prob:Pt}
Given a graph $H$ and integer $t\geq 6$, determine an upper bound on $f(H,P_t)$.
\end{problem}

\vspace{0.2cm}

\noindent {\bf (5) A hypergraph variant.}~
Recently, the first author and Wang~\cite{LiWa} studied a hypergraph version of the problem.
We consider the following three natural 3-uniform paths of length~3: the tight path $\mathcal{T}=\{v_1v_2v_3, v_2v_3v_4, v_3v_4v_5\}$, the messy path $\mathcal{M}=\{v_1v_2v_3, v_2v_3v_4, v_4v_5v_6\}$ and the loose path $\mathcal{L}=\{v_1v_2v_3, v_3v_4v_5, v_5v_6v_7\}$.
The first author and Wang in \cite{LiWa} characterized the structures of edge-colored $K_n^{(3)}$ without rainbow $\mathcal{T}$, $\mathcal{M}$ and $\mathcal{L}$, respectively.
As applications, they showed that $f(H, G)=R_2(H)$ for $G\in \{\mathcal{T}, \mathcal{M}, \mathcal{L}\}$ and infinitely many 3-uniform hypergraphs $H$.
For higher uniformity or longer paths, the situation becomes more complicated.
We propose the following related problem for further research.

\begin{problem}\label{prob:hyper}
Let $P$ be either the 3-uniform tight path or the 3-uniform loose path of length 4.
For a 3-uniform hypergraph $H$, determine whether $f(H, P)=R_3(H)$.
\end{problem}
\vspace{0.2cm}

\noindent {\bf (6) A bipartite variant.}~
In 2004, Eroh and Oellermann~\cite{ErOe} introduced a bipartite variant of the constrained Ramsey number $f(H,G)$.
Given two bipartite graphs $H$ and $G$, the {\it bipartite constrained Ramsey number} (also called {\it bipartite rainbow Ramsey number}) $h(H,G)$ is defined as the minimum integer $n$ such that, in every edge-coloring of $K_{n,n}$ with any number of colors, there is either a monochromatic copy of $H$ or a rainbow copy of $G$.
Eroh and Oellermann~\cite{ErOe} showed that $h(H,G)$ exists if and only if either $H$ is the union of a star and some isolated vertices or $G$ is a star forest.
Therefore, if $G$ is a path of length at least 3, then $h(H,G)$ exists if and only if $H$ is  the union of a star and some isolated vertices.
Eroh and Oellermann~\cite{ErOe} proved that for integers $m,n\geq 2$ with $m\leq 2n-3$, we have $h(K_{1,n}, P_{m+1})=(m-1)(n-1)+1=BR_{m-1}(K_{1,n})$.
In particular, it follows that $h(K_{1,n}, P_{4})=BR_2(K_{1,n})=2n-1$ for $n\geq 3$, and $h(K_{1,n}, P_{5})=BR_3(K_{1,n})=3n-2$ for $n\geq 4$.
One can also deduce that $h(K_{1,2}, P_{4})=3=BR_2(K_{1,2})$, $h(K_{1,2}, P_{5})=5=BR_3(K_{1,2})+1$ and $h(K_{1,3}, P_{5})=7=BR_3(K_{1,3})$ from Lemmas~\ref{le:LiWLP4} and \ref{le:LiWLP5}.
Thus $h(K_{1,n}, P_{5})=BR_3(K_{1,n})$ holds precisely when $n\geq 3$.
Moreover, by an analogous argument to Proposition~\ref{prop:isolated}, we have $h(H, P_{5})=BR_3(H)$ when $H$ is the disjoint union of $K_{1,n}$ ($n\geq 3$) and some isolated vertices.

If we intend to study the problem for other monochromatic subgraphs, we must take into account an additional restriction on the number of colors to ensure the study is well-defined.
Given two bipartite graphs $H,G$ and a positive integer $k$, let $h_k(H,G)$ be the minimum integer $n$ such that in any edge-coloring of $K_{n,n}$ with at most $k$ colors, there is either a monochromatic copy of $H$ or a rainbow copy of $G$.
We can prove the following two results related to rainbow $P_4$ and $P_5$.
Recall that $s(H), t(H), s^{\ast}(H), t^{\ast}(H)$ are parameters defined in Section~\ref{subsec:Ramsey}.

\begin{proposition}\label{prop:bipartiteP4}
Let $H$ be a bipartite graph and $k\geq 3$ be an integer.
If $H$ is connected, or if $s(H)\geq 2$ and $k\geq \frac{t(H)-1}{s(H)-1}$, then $h_k(H, P_4)=\max\{BR_2(H), k(s(H)-1)+1\}.$
\end{proposition}

\begin{proof}
For the lower bound, since a 2-edge-colored graph contains no rainbow $P_4$, we have $h_k(H, P_4)\geq BR_2(H)$.
In order to prove $h_k(H, P_4)\geq k(s(H)-1)+1$, we define an edge-colored $K_{k(s(H)-1),k(s(H)-1)}$ with partite sets $X$ and $Y$ as follows.
We partition $X$ into $k$ parts $X_1, X_2, \ldots, X_k$ with $|X_1|=|X_2|= \cdots =|X_k|=s(H)-1$, and color the edges such that $c(X_i, Y)=i$ for every $i\in [k]$.
Note that there is no rainbow $P_4$ or monochromatic $H$ in this edge-colored $K_{k(s(H)-1),k(s(H)-1)}$, so $h_k(H, P_4)\geq k(s(H)-1)+1$.

For the upper bound, let $n=\max\{BR_2(H), k(s(H)-1)+1\}$.
Note that if $H$ is connected, then $s(H)=s^{\ast}(H)$; if $k\geq \frac{t(H)-1}{s(H)-1}$, then $t(H)\leq k(s(H)-1)+1$.
We shall prove a slightly stronger statement: if $s(H)=s^{\ast}(H)$ or $t(H)\leq \max\{BR_2(H), k(s(H)-1)+1\}$, then $h_k(H, P_4)\leq n.$
For a contradiction, suppose that there exists an edge-coloring $F$ of $K_{n,n}$ with at most $k$ colors such that there is no rainbow $P_4$ or monochromatic $H$ in $F$.
Let $k'$ be the number of colors used on $E(F)$, and let $U,V$ be the partite sets of $F$.
Then $3\leq k'\leq k$, and by Lemma~\ref{le:LiWLP4}, we may assume that $U$ can be partitioned into $k'$ parts $U_1, U_2, \ldots, U_{k'}$ such that $c(U_i, V)=i$ for every $i\in [k']$.
Note that $\max_{i\in [k']}|U_i|\geq \left\lceil\frac{n}{k'}\right\rceil\geq s(H)$ and $|V|=n\geq BR_2(H) \geq t^{\ast}(H)$.
If $s(H)=s^{\ast}(H)$, then $\max_{i\in [k']}|U_i|\geq s(H)=s^{\ast}(H)$, so there is a monochromatic $K_{s^{\ast}(H),t^{\ast}(H)}$ (and thus a monochromatic $H$) in $F$, a contradiction.
If $t(H)\leq \max\{BR_2(H), k(s(H)-1)+1\}=n=|V|$, then there is a monochromatic $K_{s(H), t(H)}$ (and thus a monochromatic $H$) in $F$, a contradiction.
This completes the proof of Proposition~\ref{prop:bipartiteP4}.
\end{proof}

\begin{proposition}\label{prop:bipartiteP5}
Let $H$ be a bipartite graph with $|V(H)|\geq 4$ and $k\geq 4$ be an integer.
If $H$ is connected or $BR_2(H)+t(H)-1\leq \max\{BR_3(H), k(s(H)-1)+1\}$, then $h_k(H, P_5)=\max\{BR_3(H), k(s(H)-1)+1\}.$
\end{proposition}

\begin{proof}
For the lower bound, since a 3-edge-colored graph contains no rainbow $P_5$, we have $h_k(H, P_5)\geq BR_3(H)$.
Moreover, since $P_4\subseteq P_5$, we have $h_k(H, P_5)\geq h_k(H, P_4)\geq k(s(H)-1)+1$ by the proof of Proposition~\ref{prop:bipartiteP4}.

For the upper bound, suppose for a contradiction that there exists an edge-coloring $F$ of $K_{n,n}$ with at most $k$ colors such that there is no rainbow $P_5$ or monochromatic $H$ in $F$, where $n=\max\{BR_3(H), k(s(H)-1)+1\}$.
By the above-mentioned results related to $h(K_{1,n}, P_{5})$ and since $|V(H)|\geq 4$, we may assume that $H$ is not a disjoint union of a star and isolated vertices.
Then $s(H)\geq 2$ and $n\geq k(s(H)-1)+1\geq k+1\geq 5$.
Let $k'$ be the number of colors used on $E(F)$, and let $U,V$ be the partite sets of $F$.
Since $n\geq BR_3(H)$, we have $4\leq k'\leq k$.
Since $n\geq 5$ and $F$ contains no rainbow $P_5$, one of Lemma~\ref{le:LiWLP5}~(i) and (ii) holds.

We first assume that Lemma~\ref{le:LiWLP5}~(i) holds.
Note that $\max_{i\in [k']}|V_i|\geq \left\lceil\frac{n}{k'}\right\rceil\geq s(H)$.
If $|U_1|\geq t(H)$ or $|U_2|\geq s^{\ast}(H)$, then $F$ contains a monochromatic $K_{s(H), t(H)}$ or $K_{s^{\ast}(H),n}$.
Since $n\geq BR_3(H)\geq t^{\ast}(H)$, this implies that $F$ contains a monochromatic $H$, a contradiction.
Therefore, $n=|U_1|+|U_2|\leq t(H)-1+s^{\ast}(H)-1< BR_2(H)+t(H)-1$.
Thus we only need to consider the case that $H$ is connected.
By Lemma~\ref{le:bipar}, we have $\max\{|U_1|, |U_2|\}\geq \frac{1}{2}n\geq \frac{1}{2}BR_3(H) \geq \frac{1}{2}(3(t(H)-1)+1)\geq t(H)$.
Thus $F$ contains a monochromatic $K_{s(H), t(H)}$ (and thus a monochromatic $H$), a contradiction.

Now we assume that Lemma~\ref{le:LiWLP5}~(ii) holds.
For the case that $H$ is connected, we define a $3$-edge-coloring $F'$ of $K_{n,n}$ by recoloring all edges of colors in $\{4, \ldots, k'\}$ with color $3$.
Since $n\geq BR_3(H)$, $F'$ contains a monochromatic copy of $H$.
If this monochromatic copy of $H$ has color~1 or 2, then $F$ also contains a monochromatic $H$ of color~1 or 2, a contradiction.
If this monochromatic copy of $H$ has color~3, then since $H$ is connected, it must be contained within some $U_{i_0}\cup V_{i_0}$ for $i_0\in \{3, 4, \ldots, k'\}$.
This implies that $F$ contains a monochromatic $H$ of color $i_0$, a contradiction.
For the case $BR_2(H)+t(H)-1\leq \max\{BR_3(H), k(s(H)-1)+1\}$, we may assume that $|U_2|=\max_{2\leq i\leq k'}|U_i|$ without loss of generality.
Thus $|U_2|\geq \big\lceil\frac{n}{k'-1}\big\rceil\geq s(H)$.
Then $|V_3\cup \cdots \cup V_{k'}|\leq t(H)-1$, since otherwise there is a monochromatic $H$ of color~1 between $U_2$ and $V_3\cup \cdots \cup V_{k'}$.
Thus $|V_2|=|V|-|V_3\cup \cdots \cup V_{k'}| \geq n-(t(H)-1)\geq BR_2(H)\geq s(H)$.
Then $|U_3\cup \cdots \cup U_{k'}|\leq t(H)-1$, since otherwise there is a monochromatic $H$ of color~1 between $V_2$ and $U_3\cup \cdots \cup U_{k'}$.
Thus $|U_2|=|U|-|U_3\cup \cdots \cup U_{k'}| \geq n-(t(H)-1)\geq BR_2(H)$.
But then since $|U_2|\geq BR_2(H)$, $|V_2|\geq BR_2(H)$ and $C(U_2, V_2)\subseteq \{1,2\}$, there is a monochromatic $H$ between $U_2$ and $V_2$, a contradiction.
This completes the proof of Proposition~\ref{prop:bipartiteP5}.
\end{proof}

Regarding general disconnected bipartite graphs, the problem under consideration seems difficult to resolve.
Therefore, we formulate the following problem for future research.

\begin{problem}\label{prob:bipar}
Given a disconnected bipartite graph $H$ and integers $k\geq t-1$ with $t\in \{4,5\}$, determine whether $h_k(H,P_t)=\max\{BR_{t-2}(H), k(s(H)-1)+1\}.$
\end{problem}

\section*{Acknowledgements}
\vspace{-0.3cm}

Xihe Li is supported by the National Natural Science Foundation of China (Grant No. 12501492), Shaanxi Fundamental Science Research Project for Mathematics and Physics (No. 25JSQ043), and the Fundamental Research Funds for the Central Universities (Grant No. GK202506024).
Xiangxiang Liu is supported by the Natural Science Basic Research Plan in Shaanxi Province of China (Grant No. 2024JC-YBQN-0015).
\vspace{-0.3cm}

\section*{Declaration of competing interests}
\vspace{-0.3cm}

The authors declare that they have no competing interests.
\vspace{-0.3cm}

\section*{Data availability}
\vspace{-0.3cm}

No data were used for the research described in the article.

\begin{spacing}{0.8} 

\end{spacing}

\end{document}